\documentclass[3p,12pt,authoryear]{elsarticle}

\usepackage{amsmath,amssymb,amsfonts,amsthm}
\usepackage{booktabs}
\usepackage{array}
\usepackage{tabularx}
\usepackage{float}
\usepackage{graphicx}
\usepackage{xcolor}
\usepackage{hyperref}
\usepackage{orcidlink}
\hypersetup{hidelinks}
\usepackage[nottoc]{tocbibind}
\biboptions{round}

\newtheorem{theorem}{Theorem}[section]
\newtheorem{proposition}[theorem]{Proposition}

\newtheorem{corollary}[theorem]{Corollary}
\theoremstyle{definition}

\theoremstyle{remark}

\newcommand{\R}{\mathbb R}
\newcommand{\N}{\mathbb N}
\newcommand{\C}{\mathbb C}
\newcommand{\ez}{e_{\mathstrut 0}}
\newcommand{\ezC}{e_{\mathstrut 0}^{C}}
\newcommand{\sd}{s_{\mathstrut 0}}
\newcommand{\one}{\mathbf 1}
\newcommand{\fl}{\mathrm{fl}}
\newcommand{\diag}{\operatorname{diag}}
\newcommand{\Var}{\operatorname{Var}}
\newcolumntype{L}[1]{>{\raggedright\arraybackslash}p{#1}}

\begin{document}

\begin{frontmatter}

\title{Algebraic and FFT-Based Methods for Discrete-Time Matrix Convolutions with Applications to Semi-Markov Models}

\author[athens]{Leonidas Kordalis\corref{cor1} \orcidlink{0009-0009-1033-7700}}
\ead{lkordali@math.uoa.gr}

\author[athens,cs]{Samis Trevezas\orcidlink{0000-0003-2262-8299}}
\cortext[cor1]{Corresponding author}
\address[athens]{Department of Mathematics, National and Kapodistrian University of Athens, Athens 15784, Greece}
\address[cs]{MICS Laboratory, CentraleSup\'elec, Universit\'e Paris-Saclay, 91190 Gif-sur-Yvette, France}

\begin{abstract}
We consider finite-horizon convolution equations with matrix-valued coefficients and their use in Markov renewal computations. A sequence is inverted in a truncated noncommutative series algebra, and explicit coefficient formulae are combined with zero-padded fast Fourier transform (FFT) multiplication, Newton iteration and Gauss--Jordan elimination. We prove exactness of finite-horizon inversion, deterministic perturbation identities, and left and right a posteriori residual bounds. The FFT analysis includes transform errors and rounding in the frequency-domain matrix products. For continuous-time semi-Markov equations, an endpoint mean-value rule converts matrix Stieltjes convolutions into discrete matrix convolutions. Error estimates are obtained under bounded-variation and smoothness assumptions, and a weighted resolvent argument yields fixed-horizon convergence, with second-order convergence for smooth kernels. The same inversion framework computes transition probabilities, first-entrance distributions, reliability, availability, renewal visits and reward-type quantities. Numerical experiments examine scaling in the horizon and state dimension, residual accuracy, first-entrance probabilities, convergence against an exact Markov benchmark, and a heavy-tailed Lognormal model. The accelerated methods preserve the probabilistic calculations while reducing the cost of long-horizon convolutional inversion.
\end{abstract}


\end{frontmatter}

\section{Introduction}

Semi-Markov processes extend Markov processes by allowing the conditional sojourn-time law to depend on the present state and on the state entered at the next transition. Their transition probabilities and related functionals are governed by Markov renewal equations. Classical finite-state accounts are given by \citet{Pyke1961} and \citet{Cinlar1975}, whereas the reliability formulation is developed by \citet{LimniosOprisan2001}. In discrete time, these equations are matrix convolution equations. In continuous time, they involve matrix Stieltjes convolutions and remain meaningful when the kernel has no density.

Convolution algebra has already been used in semi-Markov dependability. \citet{Limnios1997} represented transition and dependability measures in a convolution algebra without assuming absolute continuity of the kernel. The present paper studies the finite-horizon realization of this representation. Matrix convolutional inverses are computed in a truncated series algebra; their perturbation and residual errors are controlled; and the same inversion framework is used for several Markov renewal functionals.

Numerical treatment of renewal and Stieltjes equations has followed several directions. \citet{McConalogue1981} treated convolution integrals whose densities have singularities at the origin, and \citet{Baxter1981} discussed practical numerical convolution. \citet{Xie1989} developed procedures for renewal-type integral equations, while \citet{XiePreussCui2003} analysed the errors of integration procedures for renewal equations and convolution integrals. \citet{BoehmePreussWall1991} proposed a two-point rule for Stieltjes integrals and derived error bounds in reliability problems. \citet{Tortorella2005} studied trapezoid-like and Simpson-like algorithms for Stieltjes renewal equations and related the continuous equation to discrete renewal calculations. These methods discretize the integral equation directly. Transform methods provide a different approach: \citet{Stehfest1970} introduced a numerical inverse Laplace-transform algorithm, and \citet{FinkelsteinZarudnij2002} used Laplace-transform and fast-repair approximations for availability measures.

Specific numerical methods for Markov renewal and semi-Markov equations are also available. \citet{ElkinsWortman2001} derived tight upper and lower bounds for the Markov renewal kernel, proposed a parallel algorithm, and discussed computational error and stability. \citet{HouLimniosSchon2017} proved existence and uniqueness for countable-state Markov renewal equations through an iterative scheme and included a numerical reliability application. \citet{Mercier2008} bounded continuous-time semi-Markov quantities by two discrete-time semi-Markov models and established convergence of the bounds. \citet{JanssenManca2001} applied quadrature to the nonhomogeneous transient equation and proved convergence; \citet{CorradiJanssenManca2004} developed the corresponding homogeneous construction. \citet{MouraDroguett2010} reduced the homogeneous transient problem from $s^2$ coupled integral equations to $s$ coupled equations followed by $s$ direct integrations. \citet{WuMayaLimnios2021} approximated a continuous-time semi-Markov
process by a discrete-time semi-Markov chain and bounded the
transition-matrix error; see also the published correction
\citep{WuMayaLimnios2021Correction}. \citet{WuLimnios2024} compared algebraic, truncated and iterative procedures with Laplace-based and chain-based calculations for reliability and availability.

Iterative solutions of general matrix convolution equations were considered by \citet{KilicmanAlZhour2010} through a box-convolution product. Fast Newton inversion and arithmetic with formal power series are classical in computational algebra \citep{BrentKung1978,VonZurGathenGerhard2003}. Here these tools are applied to noncommutative matrix coefficients, combined with zero-padded FFT multiplication, and related to finite-horizon Markov renewal computation and Stieltjes discretization.

Table~\ref{tab:comparison} summarizes the closest approaches and the additional elements considered here.

\begin{table}[H]
\centering
\small
\begin{tabularx}{\textwidth}{@{}L{0.19\textwidth}L{0.34\textwidth}X@{}}
\toprule
Work & Method and object & Additional element considered here\\
\midrule
\citet{Limnios1997} & Convolution-algebra representations of semi-Markov transition and dependability measures & Finite-horizon series inversion, FFT algorithms, roundoff analysis and residual estimates\\
\citet{ElkinsWortman2001} and \citet{Mercier2008} & Bounds for Markov renewal kernels and semi-Markov quantities, with parallel computation or discrete bounding models & Direct computation of the truncated convolutional inverse and a posteriori verification of that inverse\\
\citet{JanssenManca2001}, \citet{CorradiJanssenManca2004}, and \citet{MouraDroguett2010} & Quadrature schemes for transient semi-Markov equations, convergence results, and reduction of the number of coupled integral equations & Reduction of the Stieltjes product to a matrix-sequence convolution, followed by the inverse algorithms used in discrete time\\
\citet{WuMayaLimnios2021} and \citet{WuLimnios2024} & Discrete-chain approximation with error bounds, and comparison of algebraic, truncated, iterative, Laplace-based and chain-based procedures & Newton and Gauss--Jordan inversion in a truncated series algebra, finite-horizon exactness, perturbation identities and FFT forward-error control\\
\citet{KilicmanAlZhour2010} & Fixed-point iterations for general matrix convolution equations based on box convolution & Direct truncated-series inversion and its use for Markov renewal, first-entrance and reward equations\\
\bottomrule
\end{tabularx}
\caption{Comparison with selected numerical approaches to renewal and matrix convolution equations.}
\label{tab:comparison}
\end{table}

The analysis proceeds as follows. Explicit and recursive formulae are obtained for the inverse of a matrix-valued sequence, and finite-horizon inversion is identified with inversion in a quotient ring. Zero-padded FFT multiplication is combined with Newton iteration and Gauss--Jordan elimination, with operation counts in the horizon and the matrix dimension. Deterministic perturbation identities, left and right residual estimates, and a forward-error bound accounting for the FFT and the frequency-domain matrix products provide numerical controls. For continuous-time equations, an endpoint mean-value approximation converts a matrix Stieltjes convolution into a discrete matrix convolution; its deterministic error and the resulting fixed-horizon Markov renewal error are derived. The same inversion framework yields transition probabilities, first-entrance distributions, reliability, availability, renewal visits and reward-type quantities.

Section~\ref{sec:convolution} develops the sequence algebra and the accelerated inverse algorithms. Section~\ref{sec:stieltjes} treats the Stieltjes approximation. Section~\ref{sec:semimarkov} gives the Markov renewal representations and the fixed-horizon error result. Section~\ref{sec:numerics} contains the numerical study. Proofs and implementation details are given in the appendices.

\section{Matrix-valued convolution and inversion}\label{sec:convolution}

Let $\N=\{0,1,2,\ldots\}$, let $E=\{1,\ldots,s\}$, and set

\begin{equation*}
\mathcal M_{r,q}(\C):=\C^{r\times q},
\qquad
\mathcal M_s(\C):=\mathcal M_{s,s}(\C),
\qquad
\mathfrak S_{r,q}:=\{A:\N\longrightarrow\mathcal M_{r,q}(\C)\}.
\end{equation*}

We write $\mathfrak S_s:=\mathfrak S_{s,s}$. Let $I_s$ and $O_s$ denote the identity and zero matrices in $\mathcal M_s(\C)$, respectively. For $A\in\mathfrak S_{r,m}$ and $B\in\mathfrak S_{m,q}$, their convolution is

\begin{equation}\label{eq:convolution}
(A*B)(k):=\sum_{\ell=0}^{k}A(\ell)B(k-\ell),
\qquad k\in\N.
\end{equation}

The definition includes rectangular coefficient matrices, which will be used for first-entrance kernels. For square sequences, the identity element is

\begin{equation*}
\ez(0)=I_s,
\qquad
\ez(k)=O_s,
\quad k\geq1.
\end{equation*}

The constant summation sequence is denoted by $\sd$, where $\sd(k)=I_s$ for every $k\in\N$; hence

\begin{equation*}
(\sd*A)(k)=\sum_{\ell=0}^{k}A(\ell).
\end{equation*}

For a computational horizon of length $N$, let $I_N=\{0,\ldots,N-1\}$ and define

\begin{equation*}
\|A\|_{1,N}:=\sum_{k=0}^{N-1}\|A(k)\|_2,
\qquad
\|A\|_{G,N}:=
\left(\sum_{k=0}^{N-1}\|A(k)\|_F^2\right)^{1/2},
\end{equation*}

where $\|\cdot\|_2$ and $\|\cdot\|_F$ are the spectral and Frobenius matrix norms, respectively. In particular, $\|\ez\|_{1,N}=1$. The following Young-type estimates will be used repeatedly; they are finite discrete forms of Young's convolution inequality \citep{Folland1999}:

\begin{equation}\label{eq:young}
\|A*B\|_{G,N}\leq\|A\|_{1,N}\|B\|_{G,N},
\qquad
\|A*B\|_{1,N}\leq\|A\|_{1,N}\|B\|_{1,N},
\end{equation}

where the product is truncated after coefficient $N-1$. In the first estimate, the coefficient dimensions are assumed compatible with left multiplication by $A(k)$.

The correspondence

\begin{equation*}
A\longleftrightarrow P_A(x):=\sum_{k\geq0}A(k)x^k
\end{equation*}

identifies convolution with multiplication in the noncommutative ring $\mathcal M_s(\C)[[x]]$. The $n$-fold convolution power is denoted by $A^{(n)}$, with $A^{(0)}=\ez$.

\subsection{Finite Fourier computation}

Let $A$ be supported on $I_{N_A}$ and $B$ on $I_{N_B}$. Choose a power of two $L\geq N_A+N_B-1$ and extend both sequences by zero to $I_L$. Their DFTs are

\begin{equation*}
\widehat A(m)=\sum_{k=0}^{L-1}A(k)\exp\left(-2\pi\mathrm i\frac{km}{L}\right),
\qquad m\in I_L,
\end{equation*}

and similarly for $B$, where $\mathrm i^2=-1$. The linear convolution is obtained by the inverse DFT from

\begin{equation}\label{eq:fft-convolution}
A*B=\operatorname{IDFT}_L\bigl(\widehat A(m)\widehat B(m)\bigr)_{m\in I_L}.
\end{equation}

Parseval's identity gives

\begin{equation}\label{eq:parseval}
\|A\|_{G,L}=L^{-1/2}\|\widehat A\|_{G,L}.
\end{equation}

The FFT cost is

\begin{equation}\label{eq:fft-complexity}
\mathcal O(s^2L\log L+s^3L),
\end{equation}

whereas direct computation of the first $N$ coefficients requires $\mathcal O(s^3N^2)$ operations.

We next state a forward-error bound under the standard componentwise FFT model. Let $u$ be the unit roundoff. Assume that the computed forward and inverse transforms satisfy a relative Euclidean bound $\sigma_L$, with $\sigma_L=\mathcal O(u\log L)$ for a radix-two Cooley--Tukey implementation \citep{Higham2002}. Assume also that the computed frequency-domain products satisfy

\begin{equation*}
\|\fl(XY)-XY\|_F
\leq
\gamma_s\|X\|_F\|Y\|_F,
\end{equation*}

where $\fl(\cdot)$ denotes floating-point evaluation and $\gamma_s=\mathcal O(su)$ is an admissible bound for the complex matrix multiplication used by the implementation. Define

\begin{equation*}
\beta_{L,s}:=2\sigma_L+\sigma_L^2+\gamma_s(1+\sigma_L)^2,
\qquad
\delta_{L,s}:=\beta_{L,s}+\sigma_L(1+\beta_{L,s}).
\end{equation*}

The FFT convolution error is given next.

\begin{proposition}\label{prop:fft-error}
Let $C=A*B$ be computed by zero-padding to length $L$, componentwise FFTs, frequency-domain matrix multiplication, and an inverse FFT. If $C_{\fl}$ is the floating-point result, then

\begin{equation}\label{eq:fft-forward-error}
\|C-C_{\fl}\|_{G,L}
\leq
\sqrt L\,\|A\|_{G,L}\|B\|_{G,L}\,\delta_{L,s}.
\end{equation}

If $C\neq0$, the corresponding relative bound is

\begin{equation}\label{eq:fft-relative-error}
\frac{\|C-C_{\fl}\|_{G,L}}{\|C\|_{G,L}}
\leq
\kappa_{\mathrm{conv}}(A,B)\,\delta_{L,s},
\qquad
\kappa_{\mathrm{conv}}(A,B)
:=
\frac{\sqrt L\,\|A\|_{G,L}\|B\|_{G,L}}{\|A*B\|_{G,L}}.
\end{equation}

\end{proposition}

The proof is given in \ref{app:fft-error}. The factor $\kappa_{\mathrm{conv}}$ makes explicit the conditioning of the particular convolution problem; the transform error alone is of order $u\log L$.

\subsection{Convolutional inverses}

A sequence $A\in\mathfrak S_s$ is convolutionally invertible if there exists $A^{(-1)}\in\mathfrak S_s$ such that

\begin{equation*}
A*A^{(-1)}=\ez=A^{(-1)}*A.
\end{equation*}

The inverse is characterized by its zero-order coefficient.

\begin{theorem}\label{thm:inverse}
A sequence $A\in\mathfrak S_s$ is convolutionally invertible if and only if $A(0)$ is nonsingular. Let $D$ be the sequence concentrated at zero with $D(0)=A(0)^{-1}$ and set $\widetilde A=A*D$. Then

\begin{equation}\label{eq:inverse-neumann}
A^{(-1)}
=
D*\sum_{n\geq0}(\ez-\widetilde A)^{(n)}.
\end{equation}

For every $k\in\N$, the sum is finite and

\begin{equation}\label{eq:inverse-coefficient}
A^{(-1)}(k)
=
A(0)^{-1}\sum_{n=0}^{k}(\ez-\widetilde A)^{(n)}(k).
\end{equation}

\end{theorem}

\begin{proof}
Necessity follows by evaluating $A*A^{(-1)}=\ez$ at zero. Conversely, $\widetilde A(0)=I_s$, so $C:=\ez-\widetilde A$ satisfies $C(0)=O_s$. Hence $C^{(n)}(k)=O_s$ for $n>k$, and the series

\begin{equation*}
S:=\sum_{n\geq0}C^{(n)}
\end{equation*}

is finite coefficientwise. For each fixed coefficient, the products below involve only finitely many terms, and the geometric-series calculation gives

\begin{equation*}
(\ez-C)*S
=
\sum_{n\geq0}C^{(n)}-\sum_{n\geq0}C^{(n+1)}
=
\ez.
\end{equation*}

The same calculation on the right gives $S*(\ez-C)=\ez$. Thus $S$ is the two-sided inverse of $\widetilde A$. Let $D^{(-1)}$ be concentrated at zero with $D^{(-1)}(0)=A(0)$. Since $A=\widetilde A*D^{(-1)}$, associativity gives

\begin{equation*}
A*(D*S)=\ez,
\qquad
(D*S)*A=\ez.
\end{equation*}

Thus $A^{(-1)}=D*S$.
\end{proof}

The coefficient recursion is

\begin{equation}\label{eq:inverse-recursion}
A^{(-1)}(0)=A(0)^{-1},
\qquad
A^{(-1)}(k)
=-\sum_{\ell=0}^{k-1}A^{(-1)}(\ell)A(k-\ell)A(0)^{-1},
\quad k\geq1.
\end{equation}

Indeed, the coefficient of order zero in $A^{(-1)}*A=\ez$ gives $A^{(-1)}(0)=A(0)^{-1}$. For $k\geq1$, the coefficient of order $k$ is

\begin{equation*}
\sum_{\ell=0}^{k}A^{(-1)}(\ell)A(k-\ell)=O_s,
\end{equation*}

and solving for $A^{(-1)}(k)$ gives \eqref{eq:inverse-recursion}.

For finite-horizon computation, let

\begin{equation*}
\mathfrak P_{r,q,N}
:=
\left\{
\sum_{k=0}^{N-1}A_kx^k:A_k\in\mathcal M_{r,q}(\C)
\right\},
\end{equation*}

and let

\begin{equation*}
\mathfrak A_{s,N}
:=
\mathcal M_s(\C)[x]/\langle x^N\rangle.
\end{equation*}

Compatible truncated products map $\mathfrak P_{r,m,N}\times\mathfrak P_{m,q,N}$ into $\mathfrak P_{r,q,N}$, whereas $\mathfrak A_{s,N}$ is a unital associative ring. Define

\begin{equation*}
\Pi_N:\mathfrak S_{r,q}\longrightarrow\mathfrak P_{r,q,N},
\qquad
\Pi_NA:=\sum_{k=0}^{N-1}A(k)x^k.
\end{equation*}

For square coefficients, $\Pi_NA$ is identified with its class in $\mathfrak A_{s,N}$. The inverse on the finite horizon is exact in this quotient algebra.

\begin{proposition}\label{prop:finite-horizon}
If $A(0)$ is nonsingular, then

\begin{equation}\label{eq:finite-horizon}
\Pi_N\bigl(A^{(-1)}\bigr)
=
\bigl(\Pi_NA\bigr)^{-1}
\quad\text{in }\mathfrak A_{s,N}.
\end{equation}
\end{proposition}

Thus, the coefficients $A^{(-1)}(0),\ldots,A^{(-1)}(N-1)$ depend only on $A(0),\ldots,A(N-1)$. No infinite-tail truncation error is introduced on the computed horizon.

Newton inversion is performed in $\mathfrak A_{s,N}$. Let $B_0$ be concentrated at zero, with
$B_0(0)=A(0)^{-1}$. The update is

\begin{equation}\label{eq:newton}
B_{r+1}
=
B_r+B_r*(\ez-A*B_r)
=
B_r*(2\ez-A*B_r)
\pmod{x^{2^{r+1}}}.
\end{equation}

Each step doubles the number of correct coefficients. The update is classical for formal power series \citep{BrentKung1978}; its matrix-valued use here preserves the order of the noncommutative products. An alternative is Gauss--Jordan elimination applied to $P_A(x)$ over $\C[[x]]/\langle x^N\rangle$, with row pivoting and pivots whose constant coefficient is nonzero. Details are given in \ref{app:algorithms}.

The operation counts of the three inverse procedures are summarized next.

\begin{proposition}\label{prop:inverse-complexity}
Let $N$ be the number of required coefficients. Under FFT multiplication of scalar series, the recursive inverse \eqref{eq:inverse-recursion} requires

\begin{equation*}
\mathcal O(s^3N^2)
\end{equation*}

arithmetic operations. Newton inversion requires

\begin{equation*}
\mathcal O(s^2N\log N+s^3N),
\end{equation*}

whereas Gauss--Jordan elimination over $\C[[x]]/\langle x^N\rangle$ requires

\begin{equation*}
\mathcal O(s^3N\log N).
\end{equation*}

All three methods require $\mathcal O(s^2N)$ coefficient storage. For fixed $s$, the two FFT-based methods are quasi-linear in $N$.
\end{proposition}

The proof is given in \ref{app:algorithms}. The two accelerated bounds reflect different dependencies on the state dimension: Newton iteration uses matrix products at each frequency, whereas Gauss--Jordan elimination uses $\mathcal O(s^3)$ scalar series products.

\subsection{Perturbation and residual controls}

The following identities are stated in the truncated algebra $\mathfrak A_{s,N}$.

\begin{proposition}\label{prop:perturbation}
Let $q,\widetilde q\in\mathfrak A_{s,N}$ satisfy $q(0)=\widetilde q(0)=O_s$, and define

\begin{equation*}
\psi=(\ez-q)^{(-1)},
\qquad
\widetilde\psi=(\ez-\widetilde q)^{(-1)}.
\end{equation*}

Then

\begin{equation}\label{eq:resolvent-identity}
\widetilde\psi-\psi
=
\widetilde\psi*(\widetilde q-q)*\psi,
\end{equation}

and

\begin{equation}\label{eq:perturbation-bound}
\|\widetilde\psi-\psi\|_{1,N}
\leq
\|\widetilde\psi\|_{1,N}
\|\widetilde q-q\|_{1,N}
\|\psi\|_{1,N}.
\end{equation}

Let $d\geq1$, let $L,\widetilde L\in\mathfrak P_{s,d,N}$, and set $M=\psi*L$ and $\widetilde M=\widetilde\psi*\widetilde L$. Then

\begin{equation}\label{eq:functional-perturbation}
\|\widetilde M-M\|_{1,N}
\leq
\|\widetilde\psi-\psi\|_{1,N}\|L\|_{1,N}
+
\|\widetilde\psi\|_{1,N}\|\widetilde L-L\|_{1,N}.
\end{equation}

\end{proposition}

Let $A=\ez-q\in\mathfrak A_{s,N}$, set
$\psi=A^{(-1)}$, and let
$\widetilde\psi\in\mathfrak A_{s,N}$ be a computed inverse.
 Define the two residuals

\begin{equation*}
\rho_L:=\ez-A*\widetilde\psi,
\qquad
\rho_R:=\ez-\widetilde\psi*A.
\end{equation*}

\begin{proposition}\label{prop:residual}
If either $\|\rho_L\|_{1,N}<1$ or $\|\rho_R\|_{1,N}<1$, then the corresponding residual gives the a posteriori estimate

\begin{equation}\label{eq:residual-bound}
\|\psi-\widetilde\psi\|_{1,N}
\leq
\frac{\|\rho\|_{1,N}}{1-\|\rho\|_{1,N}}
\|\widetilde\psi\|_{1,N},
\end{equation}

where $\rho$ is the residual satisfying the stated condition.
\end{proposition}

The proofs of Propositions~\ref{prop:finite-horizon}--\ref{prop:residual} are given in \ref{app:algebra-proofs}.

\section{Matrix Stieltjes convolution}\label{sec:stieltjes}

Let the entries of $A=(A_{ij})$ be right-continuous functions of bounded variation satisfying $A(0)=O_s$, and let the entries of $B=(B_{ij})$ be Borel measurable and bounded on compact intervals. The integrals below are understood in the Lebesgue--Stieltjes sense. We use the convention

\begin{equation}\label{eq:stieltjes}
[A\star B]_{ij}(t)
=
\sum_{r\in E}\int_{[0,t]}B_{rj}(t-x)\,dA_{ir}(x).
\end{equation}

The first factor is the Stieltjes integrator. This convention is consistent with the semi-Markov equation $P=\overline H+Q\star P$.

Let $t_m=mh$ and set $A_h(m)=A(t_m)$ and $B_h(m)=B(t_m)$. To compute the approximation at $m=0,\ldots,N-1$, the second factor is sampled at the additional point $t_N$. Define

\begin{equation*}
\Delta A(0)=O_s,
\qquad
\Delta A(m)=A_h(m)-A_h(m-1),
\quad m\geq1,
\end{equation*}

and

\begin{equation*}
B_2(m)=\frac{B_h(m+1)+B_h(m)}{2}.
\end{equation*}

The mean-value approximation is

\begin{equation}\label{eq:stieltjes-discrete}
[A\star B](t_m)
\approx
[\Delta A*B_2](m).
\end{equation}

The deterministic error is controlled by the continuity of the second factor. Define

\begin{equation*}
\omega_B(h;T)
:=
\max_{r,j\in E}
\sup_{\substack{0\leq t,u\leq T\\ |t-u|\leq h}}
|B_{rj}(t)-B_{rj}(u)|
\end{equation*}

and

\begin{equation*}
V_A(T)
:=
\max_{i\in E}\sum_{r\in E}\Var_{[0,T]}(A_{ir}).
\end{equation*}

\begin{proposition}\label{prop:stieltjes-deterministic}
Let $E_d(m)=[A\star B](t_m)-[\Delta A*B_2](m)$ and suppose $t_m\leq T$. Then

\begin{equation}\label{eq:stieltjes-entry-error}
|E_{d,ij}(m)|
\leq
\omega_B(h;T)\sum_{r\in E}\Var_{[0,T]}(A_{ir}),
\end{equation}

and consequently

\begin{equation}\label{eq:stieltjes-frob-error}
\|E_d(m)\|_F
\leq
sV_A(T)\omega_B(h;T).
\end{equation}

If $B$ is Lipschitz on $[0,T]$ with componentwise constant $L_B(T)$, then the right-hand side is bounded by $sV_A(T)L_B(T)h$.
\end{proposition}

Under smoothness of both factors, the endpoint-average rule is of second order. Suppose that every $A_{ir}$ is absolutely continuous with density $a_{ir}\in C^1([0,T])$ and that every $B_{rj}$ belongs to $C^2([0,T])$. Define

\begin{equation*}
\begin{aligned}
M_{A,0}&:=\max_{i,r}\|a_{ir}\|_{\infty},
&
M_{A,1}&:=\max_{i,r}\|a'_{ir}\|_{\infty},\\
M_{B,1}&:=\max_{r,j}\|B'_{rj}\|_{\infty},
&
M_{B,2}&:=\max_{r,j}\|B''_{rj}\|_{\infty}.
\end{aligned}
\end{equation*}

\begin{corollary}\label{cor:stieltjes-second-order}
Under the preceding assumptions, for $t_m\leq T$,

\begin{equation}\label{eq:stieltjes-second-order}
\|E_d(m)\|_F
\leq
s^2T h^2
\left(
\frac{M_{A,1}M_{B,1}}{12}
+
\frac{M_{A,0}M_{B,2}}{4}
\right).
\end{equation}

\end{corollary}

For a semi-Markov kernel $Q$, one has $V_Q(T)\leq1$. Thus the matrix Stieltjes error is controlled directly by the regularity of the continuation term.

For $m\in I_N$, set $C^{\pi}(m):=[A\star B](t_m)$ and $C_h(m):=[\Delta A*B_2](m)$. Let $C_{h,\fl}$ be the FFT computation of $C_h$ with padding length $L$. Combining Propositions~\ref{prop:fft-error} and \ref{prop:stieltjes-deterministic} gives

\begin{equation}\label{eq:stieltjes-total-error}
\|C^{\pi}-C_{h,\fl}\|_{G,N}
\leq
\sqrt N\,sV_A(T)\omega_B(h;T)
+
\sqrt L\,\|\Delta A\|_{G,L}\|B_2\|_{G,L}\delta_{L,s}.
\end{equation}

\section{Semi-Markov equations and functionals}\label{sec:semimarkov}

Let $(J_n,S_n)_{n\geq0}$ be a homogeneous Markov renewal process on $E$, with $S_0=0$. Its kernel is

\begin{equation*}
Q_{ij}(t)
=
\mathbb P(J_{n+1}=j,S_{n+1}-S_n\leq t\mid J_n=i).
\end{equation*}

We assume $Q_{ij}(0)=0$, $\sum_{j\in E}Q_{ij}(\infty)=1$, and $S_n\to\infty$ almost surely, where $Q_{ij}(\infty):=\lim_{t\to\infty}Q_{ij}(t)$. The semi-Markov process is $Z_t=J_{N(t)}$, where $N(t)=\sup\{n:S_n\leq t\}$. Its transition function is $P(t)=(P_{ij}(t))_{i,j\in E}$, where $P_{ij}(t)=\mathbb P(Z_t=j\mid Z_0=i)$. Let

\begin{equation*}
H_i(t)=\sum_{j\in E}Q_{ij}(t),
\qquad
\overline H(t)=\diag(1-H_i(t):i\in E).
\end{equation*}

\subsection{Discrete time}

The discrete semi-Markov kernel is

\begin{equation*}
q_{ij}(k)
=
\mathbb P(J_{n+1}=j,S_{n+1}-S_n=k\mid J_n=i),
\qquad k\geq1,
\end{equation*}

with $q(0)=O_s$ and $\sum_{j\in E}\sum_{k\geq1}q_{ij}(k)=1$. Let $\alpha$ be an initial row distribution on $E$, and define

\begin{equation*}
\overline H(k)
=
\diag\left(1-\sum_{j\in E}\sum_{\ell=1}^{k}q_{ij}(\ell):i\in E\right),
\qquad k\in\N.
\end{equation*}

The Markov renewal sequence is

\begin{equation}\label{eq:renewal-sequence}
\psi=(\ez-q)^{(-1)}=\sum_{n\geq0}q^{(n)},
\qquad
\psi=\ez+q*\psi.
\end{equation}

The transition matrix sequence satisfies

\begin{equation}\label{eq:transition-discrete}
P=\overline H+q*P,
\qquad
P=\psi*\overline H.
\end{equation}

Several further quantities are obtained through the same inversion framework. Let $D\subset E$ be a target set, let $C=E\setminus D$, and let $q_{CC}\in\mathfrak S_{|C|,|C|}$ and $q_{CD}\in\mathfrak S_{|C|,|D|}$ denote the corresponding blocks. We write $\ezC$ for the identity sequence on $C$ and $\one_D$ for the vector of ones in $\R^{|D|}$. Conditioning on the first transition gives the renewal equation

\begin{equation*}
g_D=q_{CD}+q_{CC}*g_D.
\end{equation*}

Thus the first-entrance probability-mass kernel is

\begin{equation}\label{eq:first-entrance-mass}
g_D=(\ezC-q_{CC})^{(-1)}*q_{CD}.
\end{equation}

Its accumulated kernel is

\begin{equation*}
G_D(k)=\sum_{\ell=0}^{k}g_D(\ell).
\end{equation*}

For an initial distribution $\alpha_C$ on $C$,

\begin{equation}\label{eq:first-entrance-cdf}
F_D(k)=\alpha_CG_D(k)\one_D,
\qquad
S_D(k)=1-F_D(k).
\end{equation}

Reliability is obtained by taking $D$ to be the set of failed states.

For $B\subset E$, let $I_B$ be the diagonal indicator matrix and let $\one_s$ be the vector of ones in $\R^s$. Counting the renewal at time zero, the expected number of renewal visits to $B$ up to time $k$ is

\begin{equation}\label{eq:renewal-visits}
\mathcal N_B(k)
=
\alpha\sum_{\ell=0}^{k}\psi(\ell)I_B\one_s.
\end{equation}

Let $\one_{E,B}$ denote the indicator vector of $B$ in $\R^s$. The occupation probability is $A_B(k)=\alpha P(k)\one_{E,B}$. More generally, for $d\geq1$, let $r,V\in\mathfrak S_{s,d}$ and suppose that the matrix-valued reward functional $V$ satisfies $V=r+q*V$. Then

\begin{equation}\label{eq:reward}
V=\psi*r.
\end{equation}

Algorithms for hitting times and accumulated rewards in semi-Markov models are developed by \citet{SilvestrovManca2017}.

The quantities above are collected in Table~\ref{tab:functionals}. Once the corresponding inverse has been computed, the remaining operations are matrix convolutions, cumulative sums or matrix--vector products.

\begin{table}[H]
\centering
\small
\begin{tabularx}{0.93\textwidth}{@{}L{0.27\textwidth}X@{}}
\toprule
Quantity & Convolutional representation\\
\midrule
Markov renewal sequence & $\psi=(\ez-q)^{(-1)}$\\
Transition probabilities & $P=\psi*\overline H$\\
First-entrance mass in $D$ & $g_D=(\ezC-q_{CC})^{(-1)}*q_{CD}$\\
Cumulative first-entrance probability & $F_D(k)=\alpha_CG_D(k)\one_D$, where $G_D(k)=\sum_{\ell=0}^{k}g_D(\ell)$\\
Survival before entrance & $S_D(k)=1-F_D(k)$\\
Renewal visits to $B$ & $\mathcal N_B(k)=\alpha\sum_{\ell=0}^{k}\psi(\ell)I_B\one_s$\\
Occupation probability of $B$ & $A_B(k)=\alpha P(k)\one_{E,B}$\\
Reward-type equation & $V=\psi*r$ for $r,V\in\mathfrak S_{s,d}$ and $V=r+q*V$\\
\bottomrule
\end{tabularx}
\caption{Markov renewal quantities computed from convolutional inverses.}
\label{tab:functionals}
\end{table}

\subsection{Continuous time}

For the convention \eqref{eq:stieltjes}, the transition function satisfies

\begin{equation}\label{eq:transition-continuous}
P=\overline H+Q\star P.
\end{equation}

More generally, consider

\begin{equation}\label{eq:mre-continuous}
K=L+Q\star K.
\end{equation}

For the analysis, define $Q_h(m)=Q(mh)$ for every $m\in\N$. For a fixed horizon $T$ and $N_h=\lfloor T/h\rfloor$, only the values $Q_h(0),\ldots,Q_h(N_h+1)$ are required. Set

\begin{equation*}
L_h(m)=L(mh),
\qquad
\overline H_h(m)=\overline H(mh),
\qquad
0\leq m\leq N_h.
\end{equation*}

The additional value $Q_h(N_h+1)$ is used in $\Delta^+Q_h(N_h)$. Define

\begin{equation*}
\Delta Q_h(0)=O_s,
\qquad
\Delta Q_h(m)=Q_h(m)-Q_h(m-1),
\quad m\geq1,
\end{equation*}

\begin{equation*}
\Delta^+Q_h(m)=\Delta Q_h(m+1),
\qquad
\overline Q_h(m)=\frac{\Delta Q_h(m)+\Delta^+Q_h(m)}{2}.
\end{equation*}

All sequence products in the following discrete equations are truncated after coefficient $N_h$. The mean-value approximation of \eqref{eq:mre-continuous} is

\begin{equation}\label{eq:mre-discrete}
K_h
=
L_h+\overline Q_h*K_h
-
\frac12\Delta^+Q_h*(K_h(0)\ez).
\end{equation}

Since $K_h(0)=L_h(0)$, it follows that

\begin{equation}\label{eq:mre-solution}
K_h
=
(\ez-\overline Q_h)^{(-1)}*
\left[L_h-\frac12\Delta^+Q_h*(L_h(0)\ez)\right].
\end{equation}

The fixed-horizon convergence result is stated next. Define

\begin{equation*}
\omega_K(h;T)
:=
\max_{r,j\in E}
\sup_{\substack{0\leq t,u\leq T\\ |t-u|\leq h}}
|K_{rj}(t)-K_{rj}(u)|.
\end{equation*}

\begin{theorem}\label{thm:mre-global}
Let $T>0$ be fixed, set $N_h=\lfloor T/h\rfloor$, and assume that the solution $K$ of \eqref{eq:mre-continuous} is continuous on $[0,T]$. Then there exist $h_0>0$ and $C_T<\infty$, independent of $h$, such that the solution of \eqref{eq:mre-discrete} satisfies

\begin{equation}\label{eq:mre-global}
\max_{0\leq m\leq N_h}
\|K(mh)-K_h(m)\|_F
\leq
C_T\omega_K(h;T),
\qquad 0<h\leq h_0.
\end{equation}

If $K$ is Lipschitz on $[0,T]$ with componentwise constant $L_K(T)$, then

\begin{equation}\label{eq:mre-order-one}
\max_{0\leq m\leq N_h}
\|K(mh)-K_h(m)\|_F
\leq
C_TL_K(T)h.
\end{equation}

If, in addition, the entries of $Q$ are absolutely continuous with continuously differentiable densities and the entries of $K$ belong to $C^2([0,T])$, then

\begin{equation}\label{eq:mre-order-two}
\max_{0\leq m\leq N_h}
\|K(mh)-K_h(m)\|_F
=
\mathcal O(h^2).
\end{equation}

\end{theorem}

The proof, based on a weighted convolution norm and a uniform bound for $(\ez-\overline Q_h)^{(-1)}$, is given in \ref{app:mre-proof}.

A second approximation replaces the continuous kernel by the discrete kernel $\Delta Q_h$. It gives

\begin{equation}\label{eq:process-discretization}
P_h=(\ez-\Delta Q_h)^{(-1)}*\overline H_h.
\end{equation}

The two approximations \eqref{eq:mre-solution} and \eqref{eq:process-discretization} are compared below. The first is a direct Stieltjes quadrature in the family of trapezoid-like renewal-equation schemes studied by \citet{BoehmePreussWall1991} and \citet{Tortorella2005}; the second is the associated discrete-chain approximation used by \citet{WuMayaLimnios2021}. Thus the comparison concerns two distinct reductions of the continuous-time equation.

\section{Numerical study}\label{sec:numerics}

All computations were performed on a computer equipped with an AMD Ryzen 9 9955HX 16-Core Processor, under Windows~11 x64 (build 26200), using R~4.5.3 (ucrt), the default matrix-product implementation supplied with this R installation, and LAPACK~3.12.1. Elapsed wall-clock times are reported. Each method was first executed once without timing. For every model configuration, the number of repetitions in a timing batch was then doubled until the total batch time was at least $0.5$ seconds. Five such batches were timed, and the median elapsed time per invocation was retained. This adaptive protocol avoids timer-resolution zeros and applies the same timing rule to all three inverse procedures. The Monte Carlo validation of the Lognormal model was based on $10^6$ simulated trajectories with random seed $20260621$.

The implementation stores a matrix-valued sequence as an array whose third index is time. Every FFT product is zero-padded to the next power of two, and all inverse calculations retain exactly the requested $N$ coefficients. Direct convolution and the coefficient recursion are used as references on horizons for which their quadratic cost remains practical. Accuracy is assessed by discrepancies from the direct calculation and by both residuals

\begin{equation*}
\|(\ez-q)*\widetilde\psi-\ez\|_{1,N},
\qquad
\|\widetilde\psi*(\ez-q)-\ez\|_{1,N}.
\end{equation*}

The R scripts used to generate the numerical results and figures are available from the corresponding author upon reasonable request.

\subsection{Discrete-time scaling and first entrance}

We consider the three-state semi-Markov model with embedded transition matrix $P^{\mathrm{emb}}$ given by

\begin{equation*}
P^{\mathrm{emb}}
=
\begin{pmatrix}
0&1&0\\
1/5&0&4/5\\
1&0&0
\end{pmatrix}.
\end{equation*}

The states represent normal operation, degraded operation and shutdown. The admissible transitions are $1\to2$, $2\to1$, $2\to3$ and $3\to1$.

Two scaling experiments are used. In the first, $s=3$ is fixed and the horizon varies over $N=2^7,\ldots,2^{11}$; the shifted discrete Gamma parameters are given in \ref{app:parameters}. In the second, $N=256$ is fixed and $s\in\{2,4,6,8\}$. For this experiment, each state jumps uniformly to the other states and the conditional sojourn distribution for $i\to j$ is shifted Poisson with parameter $\lambda_{ij}=3+((i+j)\bmod 4)$, where the state indices are numbered from one. The direct recursion, FFT--Newton inversion and FFT--Gauss--Jordan elimination are applied to the same kernels.

Figure~\ref{fig:scaling} reports the measured dependence on the horizon and on the state dimension, while Table~\ref{tab:execution-times} gives the corresponding horizon timings. As $N$ increases from $128$ to $2048$, the coefficient recursion increases from $8.125$ ms to $2270.000$ ms, whereas FFT--Newton increases from $3.047$ ms to $40.000$ ms and FFT--Gauss--Jordan from $1.660$ ms to $18.438$ ms. At $N=2048$, the resulting speed-ups relative to the coefficient recursion are $56.8$ and $123.1$, respectively. For fixed $N=256$, FFT--Gauss--Jordan is fastest for $s=2,4,6$, whereas FFT--Newton is faster at $s=8$, reflecting the stronger dependence of Gauss--Jordan elimination on the matrix dimension. Across both experiments, the left and right residuals of the accelerated inverses range from $5.71\times10^{-15}$ to $5.18\times10^{-13}$. Each reported time is the median elapsed time per invocation obtained with the adaptive batching protocol described above.

\begin{figure}[H]
\centering
\begin{minipage}[t]{0.49\textwidth}
\centering
\includegraphics[width=\linewidth]{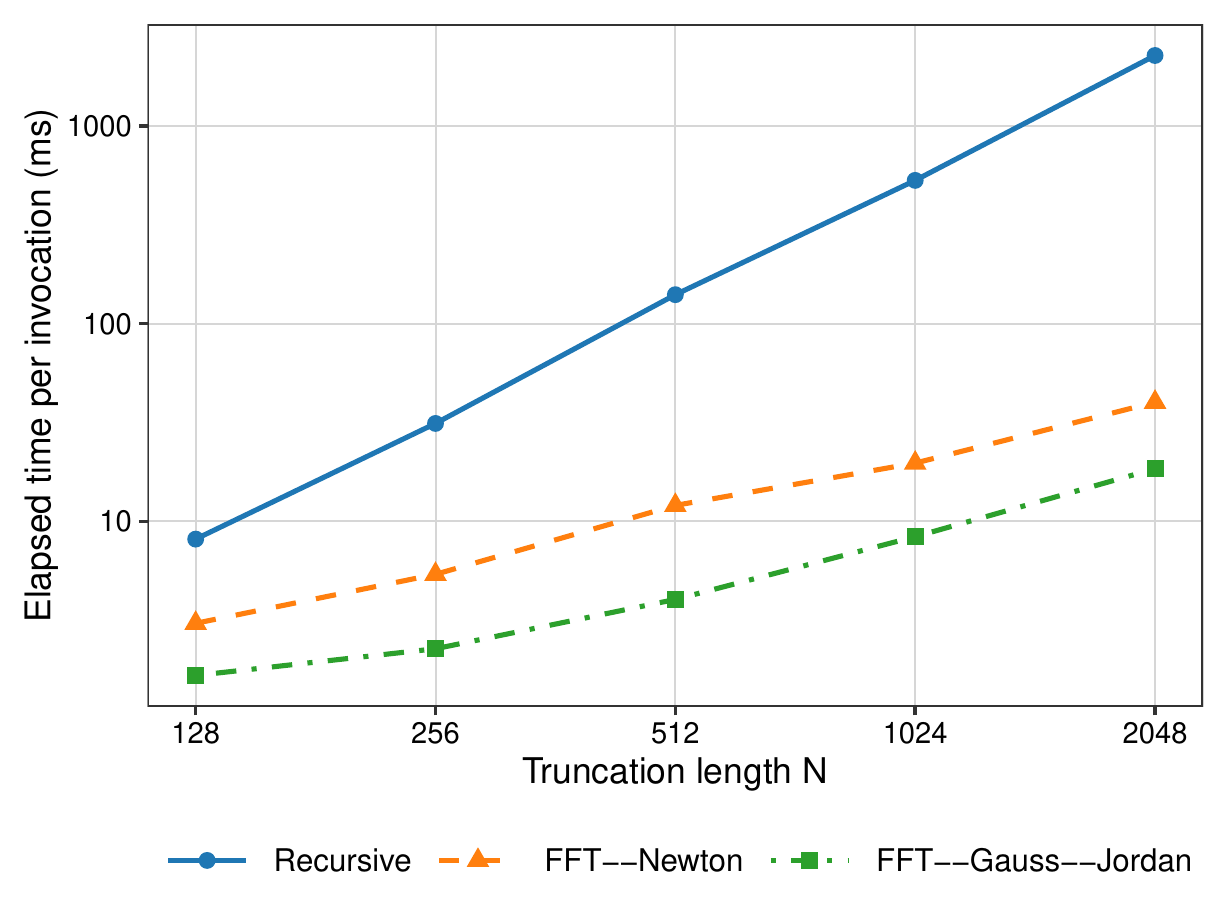}
\end{minipage}
\hfill
\begin{minipage}[t]{0.49\textwidth}
\centering
\includegraphics[width=\linewidth]{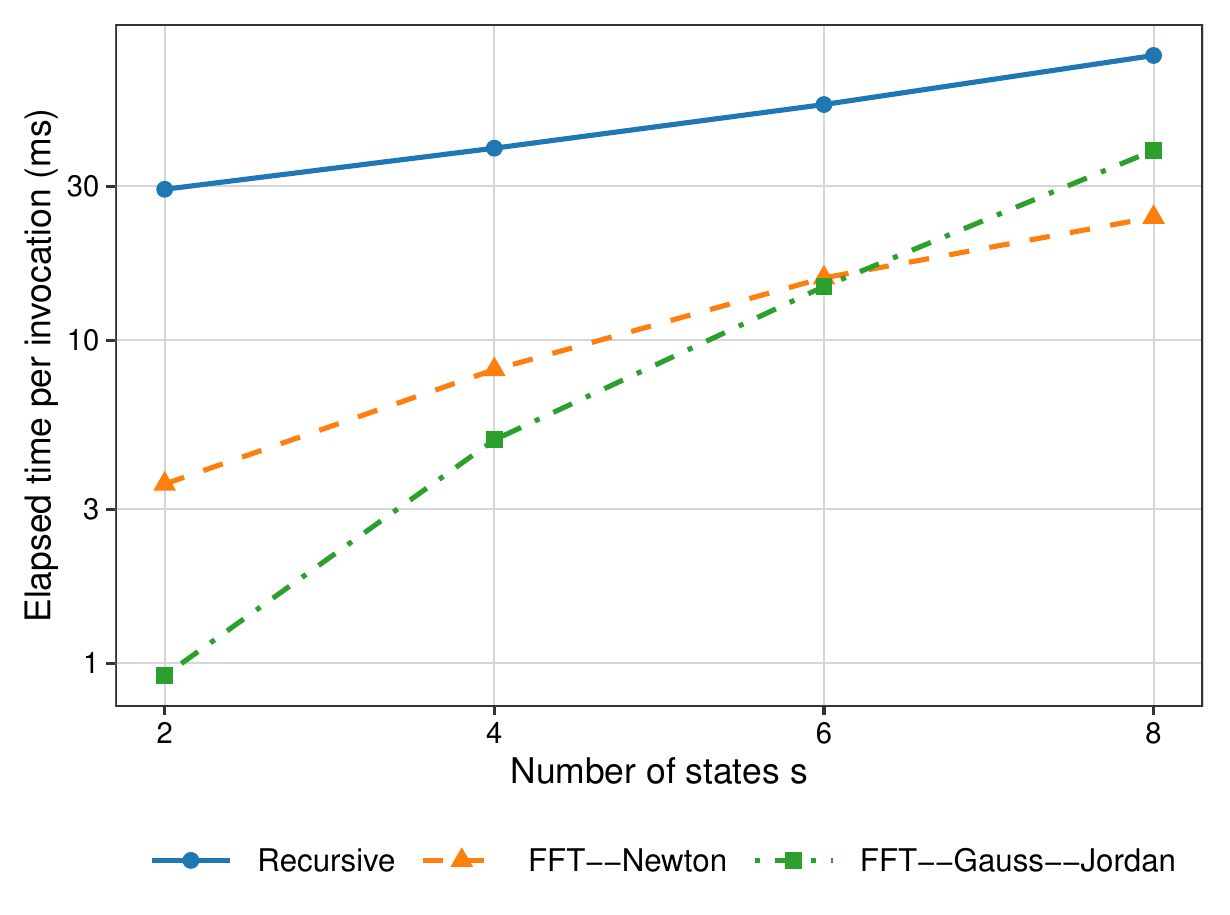}
\end{minipage}
\caption{Elapsed-time scaling for the coefficient recursion, FFT--Newton inversion and FFT--Gauss--Jordan elimination. The horizon experiment (left) uses $s=3$ and $N=2^7,\ldots,2^{11}$; the state-dimension experiment (right) uses $N=256$ and $s\in\{2,4,6,8\}$. The horizon panel uses logarithmic axes and the state-dimension panel uses a logarithmic ordinate. Each point is the median elapsed time per invocation over five adaptively sized batches, each lasting at least $0.5$ seconds after one untimed warm-up run.}
\label{fig:scaling}
\end{figure}

\begin{table}[H]
\centering
\small
\begin{tabular}{rccccc}
\toprule
$N$ & Rec. (ms) & Newton (ms) & GJ (ms) & Newton sp. & GJ sp.\\
\midrule
128  & 8.125    & 3.047  & 1.660  & 2.7  & 4.9\\
256  & 31.250   & 5.391  & 2.266  & 5.8  & 13.8\\
512  & 140.000  & 12.031 & 4.023  & 11.6 & 34.8\\
1024 & 530.000  & 19.687 & 8.359  & 26.9 & 63.4\\
2048 & 2270.000 & 40.000 & 18.438 & 56.8 & 123.1\\
\bottomrule
\end{tabular}
\caption{Median elapsed time per invocation, in milliseconds, and speed-up relative to the coefficient recursion in the horizon-scaling experiment.}
\label{tab:execution-times}
\end{table}

For the horizons at which the direct recursion is used as a reference, let $\psi_{\mathrm{rec}}$ denote the inverse obtained from the coefficient recursion \eqref{eq:inverse-recursion}, and let $\widetilde\psi$ denote the accelerated inverse under comparison. Define the scaled discrepancy

\begin{equation*}
d_N
=
\max_{0\leq k<N}
\frac{\|\widetilde\psi(k)-\psi_{\mathrm{rec}}(k)\|_F}
{\max\{\|\psi_{\mathrm{rec}}(k)\|_F,\sqrt{u}\}}.
\end{equation*}

Table~\ref{tab:inverse-accuracy} reports $d_N$ and the left and right residuals. The discrepancies are of the order of floating-point roundoff, and the residuals provide an independent check of the computed inverses. The fact that $d_N$ is smaller than the residuals is not contradictory: $d_N$ is a maximum coefficientwise discrepancy with respect to the recursive numerical reference, whereas each residual is an $\ell^1$-type norm over the complete horizon and accumulates the effects of convolution and roundoff across all coefficients.

\begin{table}[H]
\centering
\small
\begin{tabular}{ccccc}
\toprule
$N$ & Method & $d_N$ & Left residual & Right residual\\
\midrule
128 & FFT--Newton & $2.18\times10^{-15}$ & $5.71\times10^{-15}$ & $5.91\times10^{-15}$\\
128 & FFT--Gauss--Jordan & $6.54\times10^{-15}$ & $8.25\times10^{-15}$ & $7.38\times10^{-15}$\\
512 & FFT--Newton & $7.33\times10^{-15}$ & $3.23\times10^{-14}$ & $5.39\times10^{-14}$\\
512 & FFT--Gauss--Jordan & $2.02\times10^{-14}$ & $3.19\times10^{-14}$ & $2.72\times10^{-14}$\\
2048 & FFT--Newton & $2.75\times10^{-14}$ & $2.06\times10^{-13}$ & $5.18\times10^{-13}$\\
2048 & FFT--Gauss--Jordan & $2.33\times10^{-14}$ & $1.58\times10^{-13}$ & $1.51\times10^{-13}$\\
\bottomrule
\end{tabular}
\caption{Scaled discrepancy $d_N$ from the coefficient recursion and left and right residuals of the FFT--Newton and FFT--Gauss--Jordan inverses for selected horizons.}
\label{tab:inverse-accuracy}
\end{table}

For a first-entrance calculation, shifted Poisson sojourn distributions are used with

\begin{equation*}
\lambda_{12}=8,
\qquad
\lambda_{21}=5,
\qquad
\lambda_{23}=10,
\qquad
\lambda_{31}=7.
\end{equation*}

The shift guarantees $q(0)=O_s$. Let $D=\{1\}$ be the normal state, let $C=\{2,3\}$, and take $\alpha_C=(1,0)$, corresponding to the degraded state. We report the scalar first-restoration mass and survival probability

\begin{equation*}
 g_D^{\alpha}(k):=\alpha_Cg_D(k)\one_D,
 \qquad
 S_D^{\alpha}(k):=1-\alpha_CG_D(k)\one_D.
\end{equation*}

Figure~\ref{fig:first-entrance} compares direct convolution with the FFT--Newton calculation. The maximum absolute discrepancy between the two first-restoration mass functions is $1.11\times10^{-16}$, while both the left and right residuals are $7.53\times10^{-16}$. At the final reported time, the cumulative entrance probability equals one to the displayed precision.

\begin{figure}[H]
\centering
\begin{minipage}[t]{0.49\textwidth}
\centering
\includegraphics[width=\linewidth]{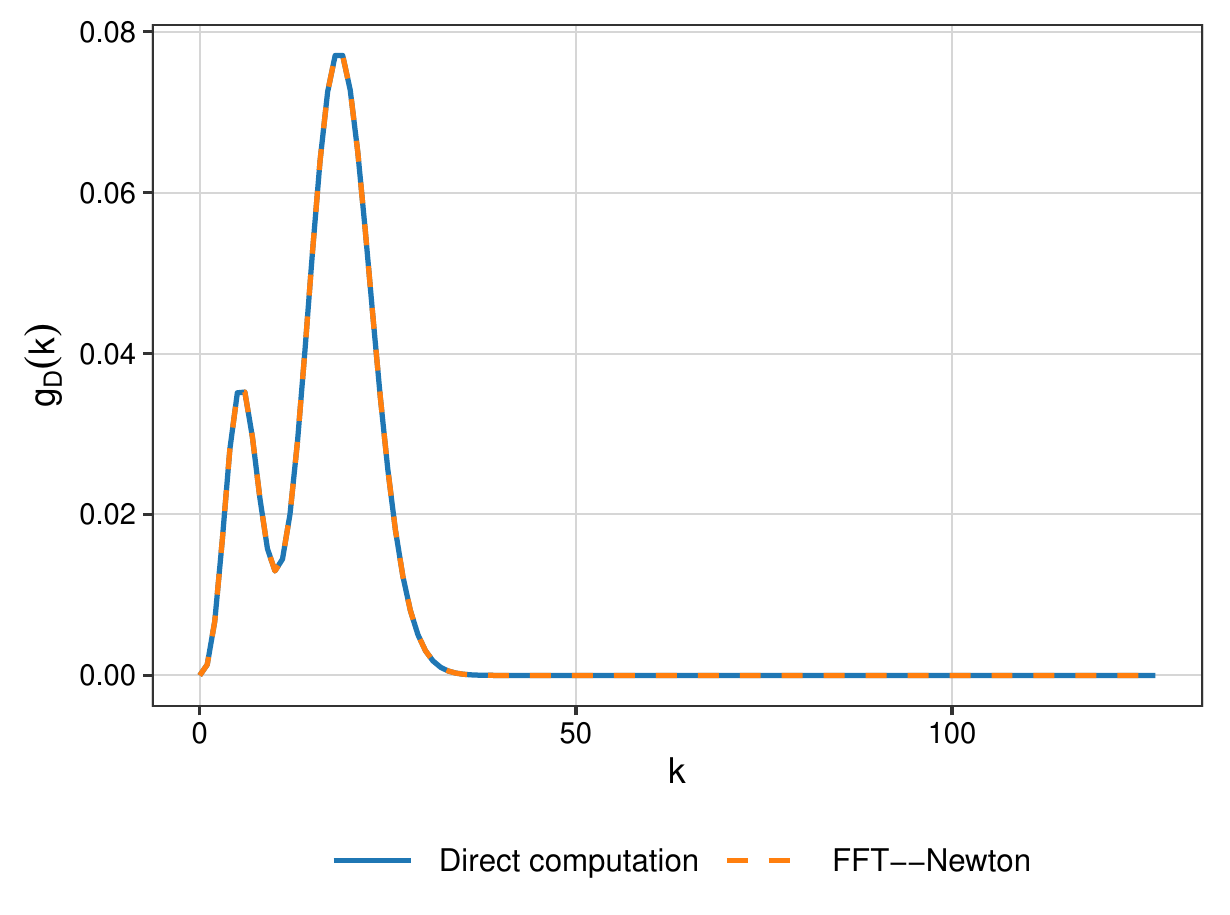}
\end{minipage}
\hfill
\begin{minipage}[t]{0.49\textwidth}
\centering
\includegraphics[width=\linewidth]{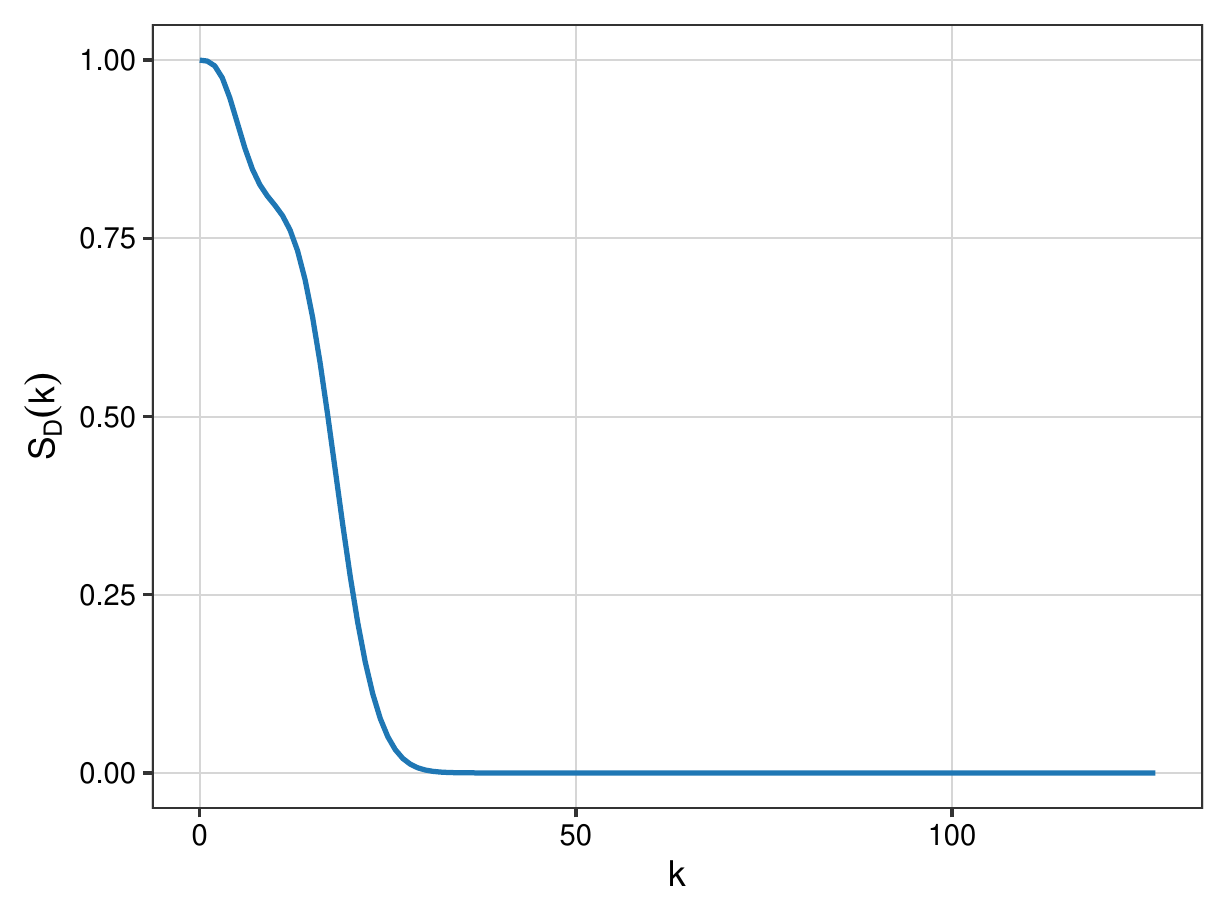}
\end{minipage}
\caption{First-restoration calculation for $D=\{1\}$, $C=\{2,3\}$, $\alpha_C=(1,0)$ and $N=128$, with shifted Poisson parameters $\lambda_{12}=8$, $\lambda_{21}=5$, $\lambda_{23}=10$ and $\lambda_{31}=7$. The left plot compares the FFT--Newton mass $g_D^{\alpha}(k)$ with the direct-convolution curve; the right plot gives the survival probability $S_D^{\alpha}(k)$.}
\label{fig:first-entrance}
\end{figure}

\subsection{Continuous-time semi-Markov model}

In this example only, we relabel the state space as $E=\{0,1,2,3\}$. We use the four-state model of \citet{LiuXingZhou2019}, with state $0$ (clean), $1$ (acquisition), $2$ (infection), and $3$ (fraud complete). The up-state set is $\mathcal U=\{0,1,2\}$ and the down-state set is $\mathcal D=\{3\}$. The exponential benchmark has generator

\begin{equation}\label{eq:generator}
\mathcal A=
\begin{pmatrix}
-0.2&0.2&0&0\\
0.01&-0.11&0.1&0\\
0.15&0.3&-0.85&0.4\\
0&0&0.5&-0.5
\end{pmatrix}.
\end{equation}

The embedded transition matrix used in every continuous-time experiment is therefore $P^{\mathrm{emb}}$, where

\begin{equation}\label{eq:embedded-continuous}
P^{\mathrm{emb}}
=
\begin{pmatrix}
0&1&0&0\\
1/11&0&10/11&0\\
3/17&6/17&0&8/17\\
0&0&1&0
\end{pmatrix}.
\end{equation}

For initial distribution $\alpha=(1,0,0,0)$, let $\mathcal A_{\mathcal U}$ be the principal submatrix of $\mathcal A$ on $\mathcal U$, let $\alpha_{\mathcal U}$ be the restriction of $\alpha$ to $\mathcal U$, let $\one_{\mathcal U}$ be the vector of ones in $\R^{|\mathcal U|}$, and let $\one_{E,\mathcal U}$ be the indicator vector of $\mathcal U$ in $\R^4$. The exact exponential reliability and availability are

\begin{equation*}
R(t)=\alpha_{\mathcal U}\exp(\mathcal A_{\mathcal U}t)\one_{\mathcal U},
\qquad
A(t)=\alpha\exp(\mathcal A t)\one_{E,\mathcal U}.
\end{equation*}

The computation uses $N=2^{15}$, $h=0.005$, and hence $T=(N-1)h=163.835$. Figure~\ref{fig:exponential} compares the exact functions with the two discrete approximations. For the displayed grid, the maximum mean-value errors are $2.53\times10^{-8}$ for availability and $7.77\times10^{-8}$ for reliability. The process-discretization errors are $6.41\times10^{-5}$ and $1.86\times10^{-4}$, respectively.

\begin{figure}[H]
\centering
\includegraphics[width=\linewidth]{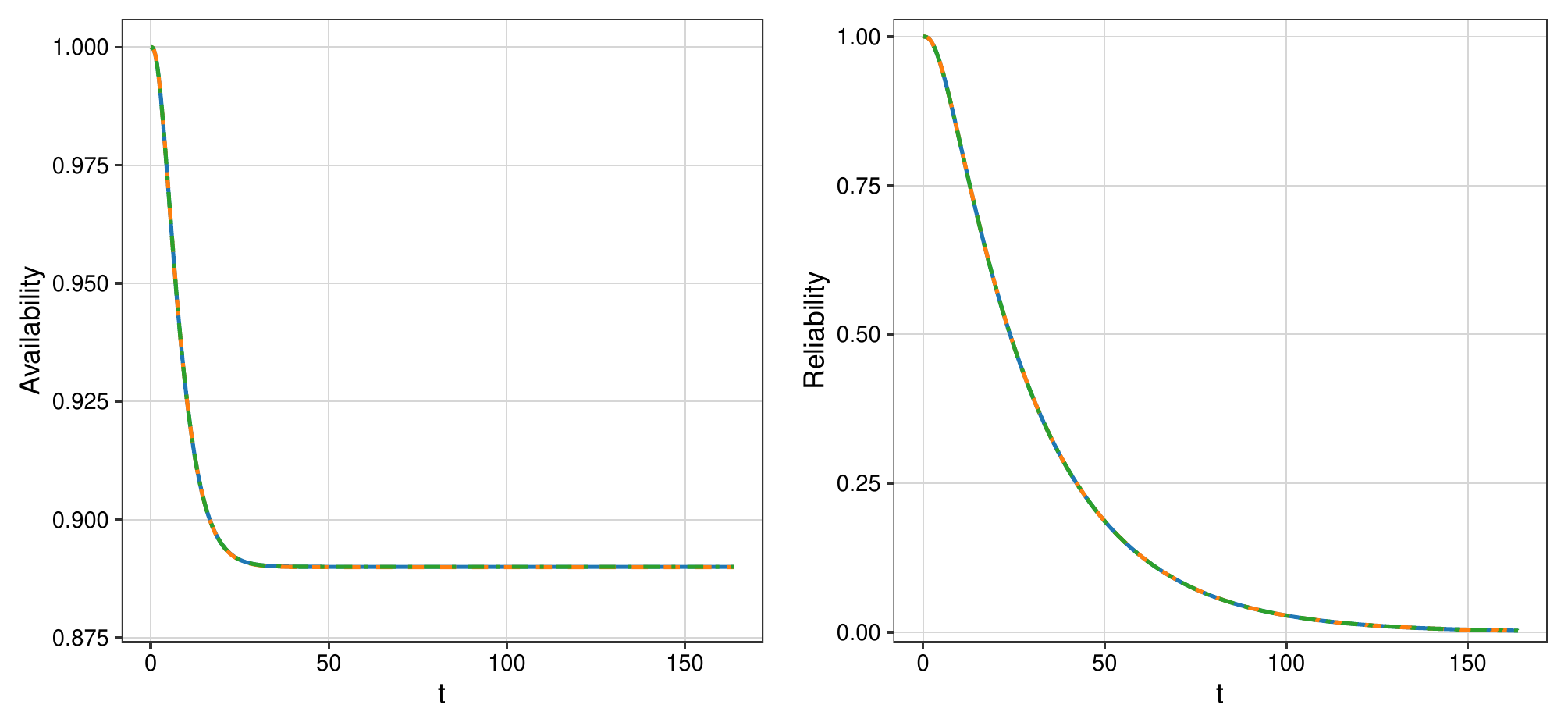}
\caption{Exponential benchmark for the four-state model on $[0,163.835]$, with $h=0.005$ and $N=2^{15}$. In both panels, the blue solid curve is the mean-value Stieltjes approximation, the orange dashed curve is the process-discretization approximation, and the green dash-dotted curve is the exact Markov solution. The availability panel focuses on values above $0.88$, whereas no artificial vertical restriction is imposed on reliability.}
\label{fig:exponential}
\end{figure}

A fixed-horizon refinement test for the mean-value approximation
\eqref{eq:mre-discrete} is given in Table~\ref{tab:convergence}. The observed order is approximately two, in agreement with the smooth-kernel conclusion \eqref{eq:mre-order-two}.

\begin{table}[H]
\centering
\small
\begin{tabular}{ccccccc}
\toprule
$h$ & $E_h^P$ & order & $E_h^R$ & order & $E_h^A$ & order\\
\midrule
0.400 & $4.7584\times10^{-3}$ & -- & $4.9635\times10^{-4}$ & -- & $1.6120\times10^{-4}$ & --\\
0.200 & $1.1923\times10^{-3}$ & 1.997 & $1.2427\times10^{-4}$ & 1.998 & $4.0345\times10^{-5}$ & 1.998\\
0.100 & $2.9732\times10^{-4}$ & 2.004 & $3.1078\times10^{-5}$ & 1.999 & $1.0089\times10^{-5}$ & 2.000\\
0.050 & $7.4285\times10^{-5}$ & 2.001 & $7.7702\times10^{-6}$ & 2.000 & $2.5224\times10^{-6}$ & 2.000\\
0.025 & $1.8571\times10^{-5}$ & 2.000 & $1.9426\times10^{-6}$ & 2.000 & $6.3062\times10^{-7}$ & 2.000\\
\bottomrule
\end{tabular}
\caption{Fixed-horizon errors on $[0,40]$ for the mean-value
approximation in the exponential benchmark. Here $E_h^P$ is the maximum Frobenius error of the transition matrix, while $E_h^R$ and $E_h^A$ are the maximum absolute reliability and availability errors. The observed orders are computed as $\log_2(E_h/E_{h/2})$.}
\label{tab:convergence}
\end{table}

Figure~\ref{fig:lognormal} gives a representative heavy-tailed case with Lognormal sojourn times. On the displayed grid, the deterministic methods differ by at most $2.20\times10^{-5}$ for availability and $3.99\times10^{-5}$ for reliability. Since no closed-form solution is available in this case, an independent pathwise validation was performed with $1{,}000{,}000$ Monte Carlo trajectories and pointwise $95\%$ normal confidence intervals. The maximum pointwise half-width is $9.80\times10^{-4}$. Table~\ref{tab:lognormal-validation} reports the maximum and root mean square discrepancies of each deterministic approximation from the Monte Carlo curves. Every maximum discrepancy is below the maximum pointwise half-width. Since the difference between the two deterministic approximations is one order of magnitude smaller than the Monte Carlo uncertainty, the simulation serves as an independent pathwise validation rather than as a means of ranking the two deterministic schemes.

\begin{table}[H]
\centering
\small
\begin{tabular}{lcccc}
\toprule
Method & $\max|\Delta A|$ & $\operatorname{RMSE}(A)$ & $\max|\Delta R|$ & $\operatorname{RMSE}(R)$\\
\midrule
Mean-value & $5.35\times10^{-4}$ & $1.46\times10^{-4}$ & $6.45\times10^{-4}$ & $1.73\times10^{-4}$\\
Process discretization & $5.14\times10^{-4}$ & $1.44\times10^{-4}$ & $6.72\times10^{-4}$ & $1.88\times10^{-4}$\\
\bottomrule
\end{tabular}
\caption{Maximum absolute discrepancy and root mean square error (RMSE) of the deterministic Lognormal curves relative to the pathwise Monte Carlo estimates based on $10^6$ trajectories.}
\label{tab:lognormal-validation}
\end{table}

\begin{figure}[H]
\centering
\begin{minipage}[t]{0.49\textwidth}
\centering
\includegraphics[width=\linewidth]{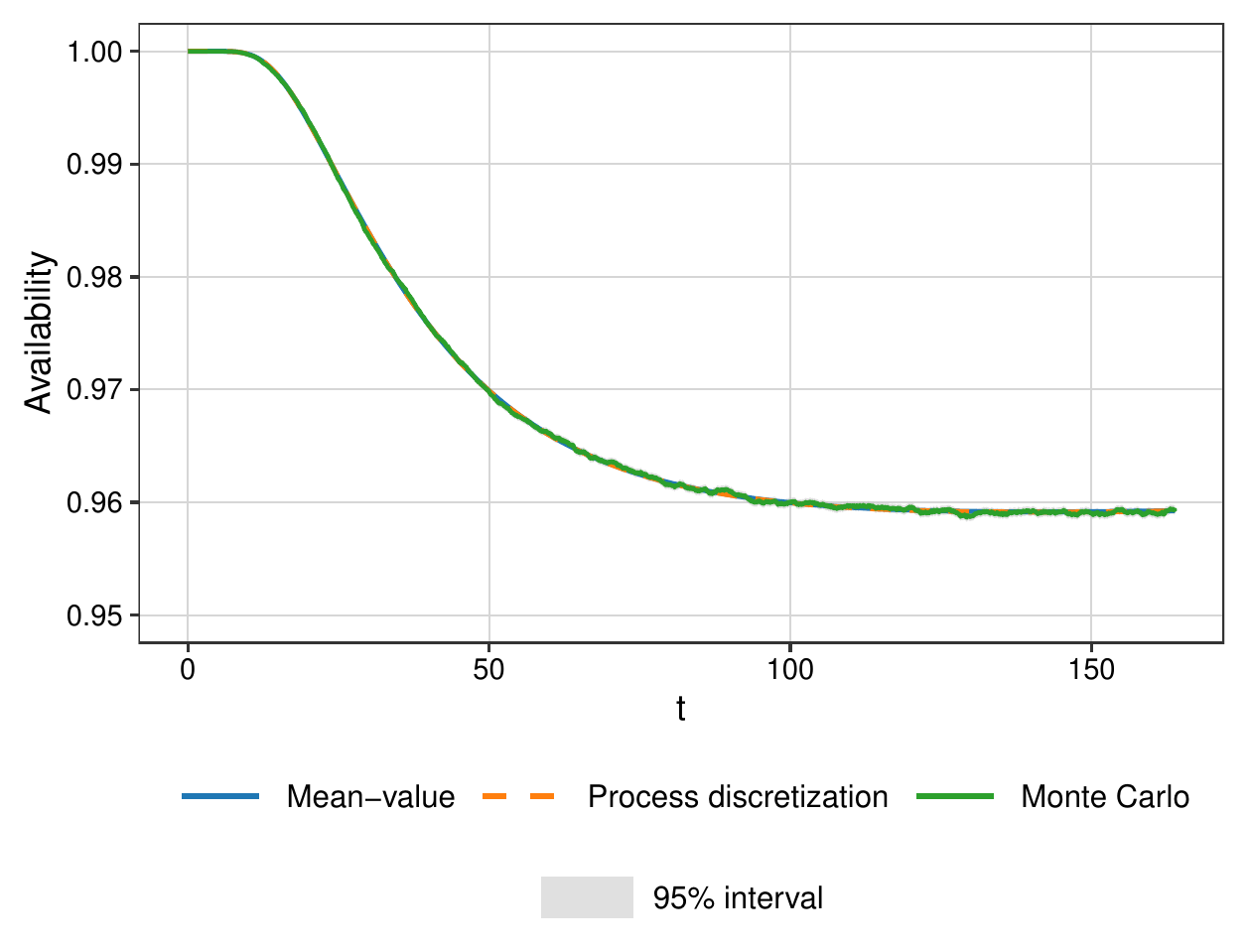}
\end{minipage}
\hfill
\begin{minipage}[t]{0.49\textwidth}
\centering
\includegraphics[width=\linewidth]{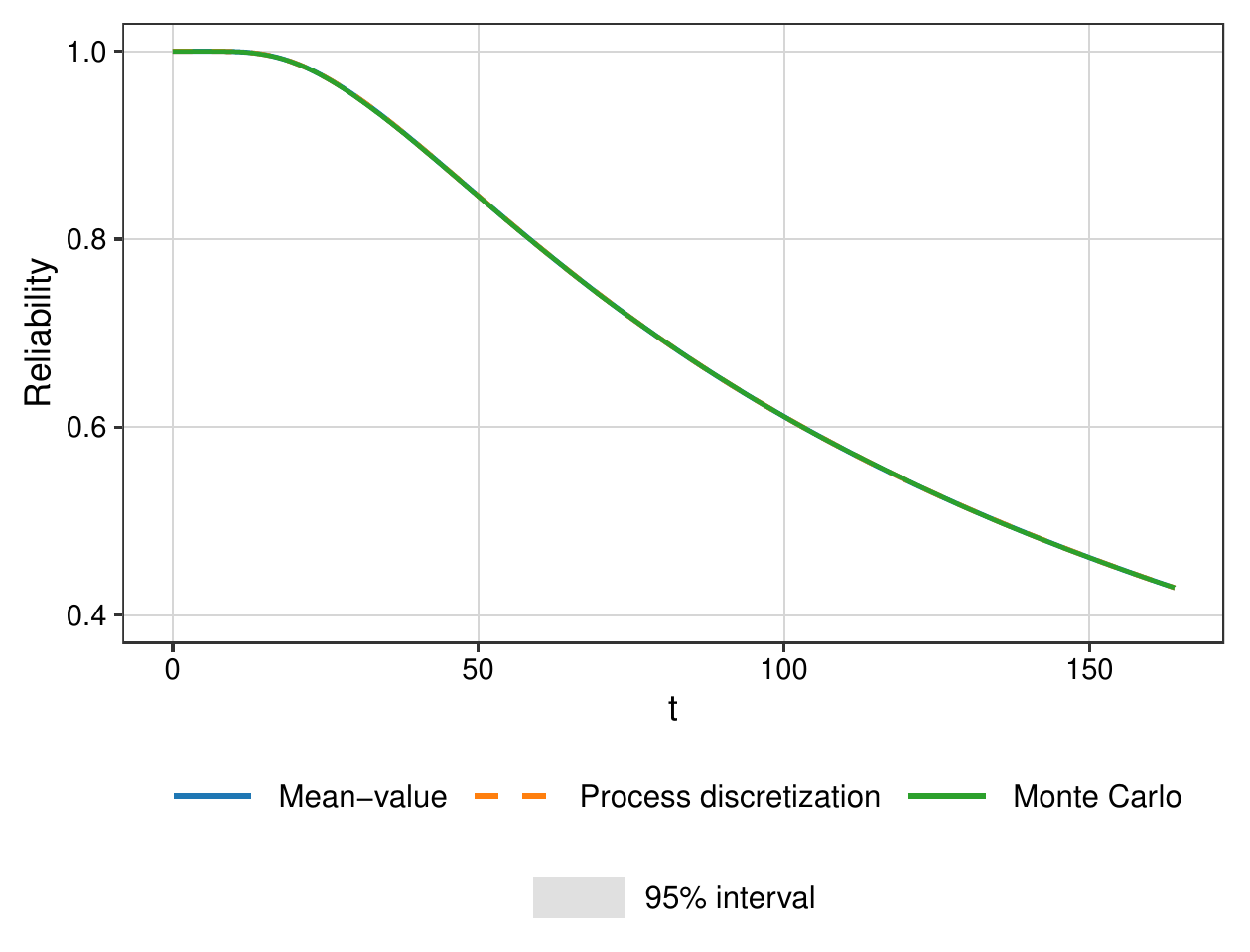}
\end{minipage}
\caption{Lognormal semi-Markov model on $[0,163.835]$, with $h=0.005$ and $N=2^{15}$. The mean-value and process-discretization approximations are compared with pathwise Monte Carlo estimates based on $1{,}000{,}000$ trajectories; the shaded regions are pointwise $95\%$ normal confidence intervals. Availability is shown on the left with ordinate restricted to $[0.95,1]$, while reliability is shown on the right over $[0.4,1]$.}
\label{fig:lognormal}
\end{figure}

\section{Discussion}\label{sec:discussion}

The finite-horizon quotient algebra separates discrete-time computation from infinite-tail truncation: the first $N$ inverse coefficients depend only on the first $N$ input coefficients. FFT multiplication then reduces the convolution cost from quadratic to quasi-linear order for fixed matrix dimension. The forward-error estimate identifies both transform error and frequency-domain matrix multiplication, while the residual bounds provide an a posteriori check of the computed inverse.

The Stieltjes approximation uses the same discrete convolution structure. Its deterministic error is controlled by the modulus of continuity of the continuation term and the row variation of the Stieltjes kernel. For semi-Markov kernels, this variation is bounded probabilistically. The weighted-resolvent argument propagates the local error through the Markov renewal equation and yields convergence on fixed time intervals.

The numerical calculations illustrate three distinct aspects: scaling of the inverse algorithms, fixed-horizon convergence against an exact Markov benchmark, and computation of a first-entrance distribution in addition to reliability and availability. At the largest tested horizon, FFT--Newton and FFT--Gauss--Jordan are respectively $56.8$ and $123.1$ times faster than the coefficient recursion, while the state-dimension experiment shows the stronger dependence of Gauss--Jordan elimination on $s$. In the Lognormal model, the maximum discrepancies from the Monte Carlo curves remain below the maximum pointwise confidence half-width, and the two deterministic schemes agree much more closely with each other than the Monte Carlo uncertainty permits one to distinguish. Further applications can use the same inversion framework for occupation, renewal-visit, and reward quantities.

\section*{Statements and declarations}

\paragraph{Funding}
The authors declare that no funds, grants, or other support were received during the preparation of this manuscript.

\paragraph{Competing interests}
The authors have no relevant financial or non-financial interests to disclose.

\paragraph{Author contributions}
Both authors contributed to the conception, mathematical development, numerical methodology, verification and preparation of the manuscript. Both authors read and approved the final manuscript.

\paragraph{Data and code availability}
No external datasets were used in this study. The R scripts used to generate the numerical results and figures are available from the corresponding author upon reasonable request.

\appendix

\section{Algebraic proofs}\label{app:algebra-proofs}

We first prove Proposition~\ref{prop:finite-horizon}.

\begin{proof}
For square sequences, the map $\Pi_N$ is the canonical ring homomorphism from $\mathfrak S_s$ onto $\mathfrak A_{s,N}$. Indeed, for every $k<N$, the coefficient of $x^k$ in $(\Pi_NA)(\Pi_NB)$ is

\begin{equation*}
\sum_{\ell=0}^{k}A(\ell)B(k-\ell)=(A*B)(k),
\end{equation*}

so $\Pi_N(A*B)=(\Pi_NA)(\Pi_NB)$ modulo $x^N$. Applying $\Pi_N$ to $A*A^{(-1)}=\ez=A^{(-1)}*A$ gives

\begin{equation*}
(\Pi_NA)\Pi_N(A^{(-1)})
=
I_s
=
\Pi_N(A^{(-1)})(\Pi_NA).
\end{equation*}

in $\mathfrak A_{s,N}$. Uniqueness of the inverse in a unital ring proves \eqref{eq:finite-horizon}.
\end{proof}

We next prove Proposition~\ref{prop:perturbation}.

\begin{proof}
Set $A=\ez-q$ and $\widetilde A=\ez-\widetilde q$. Since $\psi=A^{(-1)}$ and $\widetilde\psi=\widetilde A^{(-1)}$, associativity and the two-sided inverse property give

\begin{equation*}
\begin{aligned}
\widetilde\psi-\psi
&=\widetilde\psi*A*\psi-\widetilde\psi*\widetilde A*\psi\\
&=\widetilde\psi*(A-\widetilde A)*\psi\\
&=\widetilde\psi*(\widetilde q-q)*\psi.
\end{aligned}
\end{equation*}

Applying the second estimate in \eqref{eq:young} twice yields

\begin{equation*}
\|\widetilde\psi-\psi\|_{1,N}
\leq
\|\widetilde\psi\|_{1,N}
\|\widetilde q-q\|_{1,N}
\|\psi\|_{1,N},
\end{equation*}

which proves \eqref{eq:perturbation-bound}. Moreover,

\begin{equation*}
\widetilde M-M
=(\widetilde\psi-\psi)*L+\widetilde\psi*(\widetilde L-L).
\end{equation*}

A further application of \eqref{eq:young} to the two terms gives \eqref{eq:functional-perturbation}.
\end{proof}

We finally prove Proposition~\ref{prop:residual}.

\begin{proof}
Let $A=\ez-q$ and $\psi=A^{(-1)}$. Consider first the right residual $\rho_R=\ez-\widetilde\psi*A$. Then

\begin{equation*}
\widetilde\psi*A=\ez-\rho_R.
\end{equation*}

Multiplication on the right by $\psi$ gives

\begin{equation*}
\widetilde\psi=(\ez-\rho_R)*\psi.
\end{equation*}

If $\|\rho_R\|_{1,N}<1$, the Neumann series converges in $\mathfrak A_{s,N}$ and

\begin{equation*}
(\ez-\rho_R)^{(-1)}
=
\sum_{n\geq0}\rho_R^{(n)},
\qquad
\|(\ez-\rho_R)^{(-1)}\|_{1,N}
\leq
\frac{1}{1-\|\rho_R\|_{1,N}}.
\end{equation*}

Consequently,

\begin{equation*}
\psi=(\ez-\rho_R)^{(-1)}*\widetilde\psi
\end{equation*}

and

\begin{equation*}
\begin{aligned}
\psi-\widetilde\psi
&=\left[(\ez-\rho_R)^{(-1)}-\ez\right]*\widetilde\psi\\
&=(\ez-\rho_R)^{(-1)}*\rho_R*\widetilde\psi.
\end{aligned}
\end{equation*}

The second estimate in \eqref{eq:young} therefore gives

\begin{equation*}
\|\psi-\widetilde\psi\|_{1,N}
\leq
\frac{\|\rho_R\|_{1,N}}{1-\|\rho_R\|_{1,N}}
\|\widetilde\psi\|_{1,N}.
\end{equation*}

For the left residual $\rho_L=\ez-A*\widetilde\psi$, one has

\begin{equation*}
A*\widetilde\psi=\ez-\rho_L.
\end{equation*}

Multiplication on the left by $\psi$ gives

\begin{equation*}
\widetilde\psi=\psi*(\ez-\rho_L).
\end{equation*}

If $\|\rho_L\|_{1,N}<1$, then

\begin{equation*}
\psi=\widetilde\psi*(\ez-\rho_L)^{(-1)}
\end{equation*}

and hence

\begin{equation*}
\psi-\widetilde\psi
=
\widetilde\psi*\rho_L*(\ez-\rho_L)^{(-1)}.
\end{equation*}

The same norm estimate proves \eqref{eq:residual-bound} with $\rho=\rho_L$.
\end{proof}

\section{FFT forward-error calculation}\label{app:fft-error}

We prove Proposition~\ref{prop:fft-error}.

\begin{proof}
Write the computed transforms as $\widehat A_{\fl}=\widehat A+E_A$ and $\widehat B_{\fl}=\widehat B+E_B$, with

\begin{equation*}
\|E_A\|_{G,L}\leq\sigma_L\|\widehat A\|_{G,L},
\qquad
\|E_B\|_{G,L}\leq\sigma_L\|\widehat B\|_{G,L}.
\end{equation*}

Before rounding the matrix products, the perturbation is

\begin{equation*}
\widehat A E_B+E_A\widehat B+E_AE_B.
\end{equation*}

For any two matrix arrays $X=(X(m))_{m\in I_L}$ and $Y=(Y(m))_{m\in I_L}$,

\begin{equation*}
\|XY\|_{G,L}^2
=
\sum_{m\in I_L}\|X(m)Y(m)\|_F^2
\leq
\|X\|_{G,L}^2\|Y\|_{G,L}^2.
\end{equation*}

Using this estimate for the three terms and the bounds on $E_A$ and $E_B$ gives

\begin{equation*}
\|(\widehat A+E_A)(\widehat B+E_B)-\widehat A\widehat B\|_{G,L}
\leq
(2\sigma_L+\sigma_L^2)\|\widehat A\|_{G,L}\|\widehat B\|_{G,L}.
\end{equation*}

The standard matrix multiplication model adds at most

\begin{equation*}
\gamma_s(1+\sigma_L)^2\|\widehat A\|_{G,L}\|\widehat B\|_{G,L}.
\end{equation*}

Let $Z(m)=\widehat A(m)\widehat B(m)$ and let $Z_{\fl}$ denote the computed frequency-domain product. The preceding bounds give

\begin{equation*}
\|Z_{\fl}-Z\|_{G,L}
\leq
\beta_{L,s}\|\widehat A\|_{G,L}\|\widehat B\|_{G,L},
\end{equation*}

and hence

\begin{equation*}
\|Z_{\fl}\|_{G,L}
\leq
(1+\beta_{L,s})\|\widehat A\|_{G,L}\|\widehat B\|_{G,L}.
\end{equation*}

Let $\widetilde C=\operatorname{IDFT}_L(Z_{\fl})$. Parseval's identity and the inverse-transform error model yield

\begin{equation*}
\begin{aligned}
\|C-C_{\fl}\|_{G,L}
&\leq
\|C-\widetilde C\|_{G,L}
+
\|\widetilde C-C_{\fl}\|_{G,L}\\
&\leq
L^{-1/2}
\left[\beta_{L,s}+\sigma_L(1+\beta_{L,s})\right]
\|\widehat A\|_{G,L}\|\widehat B\|_{G,L}.
\end{aligned}
\end{equation*}

A second application of \eqref{eq:parseval} proves \eqref{eq:fft-forward-error}; division by $\|C\|_{G,L}$ proves \eqref{eq:fft-relative-error}.
\end{proof}

\section{Stieltjes and Markov renewal error bounds}\label{app:mre-proof}

We first prove Proposition~\ref{prop:stieltjes-deterministic}.

\begin{proof}
For fixed $i,j\in E$ and $t_m=mh$, write the error as

\begin{equation*}
E_{d,ij}(m)
=
\sum_{r\in E}\sum_{\ell=1}^{m}
\int_{(t_{\ell-1},t_\ell]}
\left[
B_{rj}(t_m-x)
-
\frac{B_{rj}(t_m-t_{\ell-1})+B_{rj}(t_m-t_\ell)}{2}
\right]dA_{ir}(x).
\end{equation*}

For $x\in(t_{\ell-1},t_\ell]$, both $t_m-x$ and the two endpoint arguments in the preceding display differ by at most $h$. Therefore,

\begin{equation*}
\left|
B_{rj}(t_m-x)
-
\frac{B_{rj}(t_m-t_{\ell-1})+B_{rj}(t_m-t_\ell)}{2}
\right|
\leq
\omega_B(h;T).
\end{equation*}

The variation inequality for Lebesgue--Stieltjes integrals now gives

\begin{equation*}
|E_{d,ij}(m)|
\leq
\omega_B(h;T)
\sum_{r\in E}\operatorname{Var}_{[0,T]}(A_{ir}),
\end{equation*}

which proves \eqref{eq:stieltjes-entry-error}. Since the same bound holds for all $s^2$ entries, \eqref{eq:stieltjes-frob-error} follows. The Lipschitz statement is obtained from $\omega_B(h;T)\leq L_B(T)h$.
\end{proof}

We next prove Corollary~\ref{cor:stieltjes-second-order}.

\begin{proof}
Fix an interval $[a,b]$ of length $h$ and let $c=(a+b)/2$. For $f\in C^2([a,b])$, Taylor's formula gives

\begin{equation*}
f(x)-\frac{f(a)+f(b)}{2}
=
f'(c)(x-c)+R_f(x),
\qquad
|R_f(x)|\leq\frac14\|f''\|_{\infty}h^2.
\end{equation*}

If $a_0\in C^1([a,b])$, then

\begin{equation*}
\left|\int_a^b(x-c)a_0(x)\,dx\right|
=
\left|\int_a^b(x-c)[a_0(x)-a_0(c)]\,dx\right|
\leq
\frac{h^3}{12}\|a_0'\|_{\infty}.
\end{equation*}

Moreover,

\begin{equation*}
\left|\int_a^bR_f(x)a_0(x)\,dx\right|
\leq
\frac{h^3}{4}\|f''\|_{\infty}\|a_0\|_{\infty}.
\end{equation*}

Apply these estimates with $f(x)=B_{rj}(t_m-x)$ and $a_0(x)=a_{ir}(x)$ on every grid interval. Summing over at most $T/h$ intervals and over $r\in E$ gives the entrywise bound

\begin{equation*}
|E_{d,ij}(m)|
\leq
sT h^2
\left(
\frac{M_{A,1}M_{B,1}}{12}
+
\frac{M_{A,0}M_{B,2}}{4}
\right).
\end{equation*}

Taking the Frobenius norm over the $s^2$ entries proves \eqref{eq:stieltjes-second-order}.
\end{proof}

We next prove Theorem~\ref{thm:mre-global}.

\begin{proof}
Let

\begin{equation*}
K_h^{\pi}(m):=K(mh),
\qquad
0\leq m\leq N_h,
\end{equation*}

and set $e_h=K_h^{\pi}-K_h$. Define the consistency error by

\begin{equation*}
\tau_h
=
K_h^{\pi}-L_h-\overline Q_h*K_h^{\pi}
+\frac12\Delta^+Q_h*(K(0)\ez).
\end{equation*}

At $m=0$, the continuous equation gives $K(0)=L(0)$ because $Q(0)=O_s$. In the discrete equation, $\overline Q_h(0)=\Delta Q_h(1)/2$ and $\Delta^+Q_h(0)=\Delta Q_h(1)$; the two terms containing $K_h(0)$ cancel, and hence $K_h(0)=L_h(0)=K(0)$. Subtracting \eqref{eq:mre-discrete} from the definition of $\tau_h$ therefore gives

\begin{equation}\label{eq:error-equation-app}
(\ez-\overline Q_h)*e_h=\tau_h.
\end{equation}

For $0\leq m\leq N_h$, direct expansion gives

\begin{equation*}
\left[
\overline Q_h*K_h^{\pi}
-
\frac12\Delta^+Q_h*(K(0)\ez)
\right](m)
=
\frac12
\sum_{\ell=1}^{m}
\Delta Q_h(\ell)
\left[
K((m-\ell+1)h)+K((m-\ell)h)
\right].
\end{equation*}

Thus $\tau_h(m)$ is exactly the error of the Stieltjes approximation \eqref{eq:stieltjes-discrete} with $A=Q$ and $B=K$. Since every $Q_{ij}$ is nondecreasing and

\begin{equation*}
\sum_{j\in E}\operatorname{Var}_{[0,T]}(Q_{ij})
=
\sum_{j\in E}Q_{ij}(T)
\leq1,
\end{equation*}

Proposition~\ref{prop:stieltjes-deterministic} yields

\begin{equation}\label{eq:local-error-app}
\|\tau_h(m)\|_F
\leq
s\omega_K(h;T),
\qquad
0\leq m\leq N_h.
\end{equation}

It remains to obtain a bound for the discrete resolvent which is uniform with respect to $h$. For $\lambda>0$, define

\begin{equation*}
r_\lambda
:=
\max_{i\in E}
\sum_{j\in E}
\int_{(0,\infty)}e^{-\lambda t}\,dQ_{ij}(t).
\end{equation*}

The measure $\sum_jdQ_{ij}$ has total mass one and no atom at zero. Hence dominated convergence gives $r_\lambda\to0$ as $\lambda\to\infty$. Choose $\lambda>0$ such that $r_\lambda<1$, and then choose $h_0>0$ such that

\begin{equation*}
\varrho
:=
\frac{1+e^{\lambda h_0}}{2}r_\lambda
<1.
\end{equation*}

For a matrix sequence $A$, define

\begin{equation*}
\|A\|_{\lambda,h}
:=
\max_{i\in E}
\sum_{j\in E}
\sum_{k\geq0}e^{-\lambda hk}|A_{ij}(k)|.
\end{equation*}

To verify submultiplicativity, let $C=A*B$. For every $i$,

\begin{equation*}
\begin{aligned}
\sum_j\sum_{n\geq0}e^{-\lambda hn}|C_{ij}(n)|
&\leq
\sum_{r}\sum_{k\geq0}e^{-\lambda hk}|A_{ir}(k)|
\sum_j\sum_{\ell\geq0}e^{-\lambda h\ell}|B_{rj}(\ell)|\\
&\leq
\|B\|_{\lambda,h}
\sum_r\sum_{k\geq0}e^{-\lambda hk}|A_{ir}(k)|.
\end{aligned}
\end{equation*}

Taking the maximum over $i$ gives

\begin{equation*}
\|A*B\|_{\lambda,h}
\leq
\|A\|_{\lambda,h}\|B\|_{\lambda,h}.
\end{equation*}

For the increments of the semi-Markov kernel, nonnegativity gives

\begin{equation*}
\begin{aligned}
\|\Delta Q_h\|_{\lambda,h}
&=
\max_i\sum_j\sum_{k\geq1}e^{-\lambda hk}
\left[Q_{ij}(kh)-Q_{ij}((k-1)h)\right]\\
&=
\max_i\sum_j\sum_{k\geq1}
\int_{((k-1)h,kh]}e^{-\lambda hk}\,dQ_{ij}(t)\\
&\leq
\max_i\sum_j\int_{(0,\infty)}e^{-\lambda t}\,dQ_{ij}(t)
=r_\lambda,
\end{aligned}
\end{equation*}

because $t\leq kh$ on the $k$th interval. Similarly, after the change of index $r=k+1$,

\begin{equation*}
\begin{aligned}
\|\Delta^+Q_h\|_{\lambda,h}
&=
\max_i\sum_j\sum_{k\geq0}e^{-\lambda hk}\Delta Q_{h,ij}(k+1)\\
&=e^{\lambda h}
\max_i\sum_j\sum_{r\geq1}e^{-\lambda hr}\Delta Q_{h,ij}(r)\\
&\leq e^{\lambda h}r_\lambda.
\end{aligned}
\end{equation*}

Thus, for $0<h\leq h_0$,

\begin{equation*}
\|\overline Q_h\|_{\lambda,h}
\leq
\frac{1+e^{\lambda h}}{2}r_\lambda
\leq\varrho<1.
\end{equation*}

The Neumann series therefore converges in the weighted convolution norm and gives

\begin{equation*}
U_h:=(\ez-\overline Q_h)^{(-1)}
=
\sum_{n\geq0}\overline Q_h^{(n)},
\qquad
\|U_h\|_{\lambda,h}
\leq
\frac{1}{1-\varrho}.
\end{equation*}

Let $\Pi_{N_h+1}$ denote the projection onto the coefficients of orders
$0,\ldots,N_h$. By Proposition~\ref{prop:finite-horizon}, equivalently by
projecting the preceding identity onto the truncated algebra, the coefficients
$U_h(0),\ldots,U_h(N_h)$ coincide with those of the inverse of
$\Pi_{N_h+1}(\ez-\overline Q_h)$ in the truncated algebra used in
\eqref{eq:mre-discrete}. Since $\overline Q_h$ has nonnegative coefficients,
all coefficients of $U_h$ are nonnegative. If $mh\leq T$, then

\begin{equation*}
\begin{aligned}
\sum_{\ell=0}^{m}\|U_h(\ell)\|_F
&\leq
\sum_{\ell=0}^{m}\sum_{i,j\in E}U_{h,ij}(\ell)\\
&\leq
e^{\lambda T}
\sum_{i\in E}\sum_{j\in E}\sum_{\ell\geq0}
e^{-\lambda h\ell}U_{h,ij}(\ell)\\
&\leq
s e^{\lambda T}\|U_h\|_{\lambda,h}
\leq
\frac{s e^{\lambda T}}{1-\varrho}.
\end{aligned}
\end{equation*}

Since $\|U_h(\ell)\|_2\leq\|U_h(\ell)\|_F$, the cumulative Frobenius-norm estimate obtained above also controls the spectral norms occurring in the convolution estimate. Solving \eqref{eq:error-equation-app} gives $e_h=U_h*\tau_h$. Hence, using \eqref{eq:local-error-app},

\begin{equation*}
\begin{aligned}
\|e_h(m)\|_F
&\leq
\sum_{\ell=0}^{m}\|U_h(\ell)\|_2
\|\tau_h(m-\ell)\|_F\\
&\leq
\frac{s^2e^{\lambda T}}{1-\varrho}
\omega_K(h;T).
\end{aligned}
\end{equation*}

This proves \eqref{eq:mre-global}, with $C_T=s^2e^{\lambda T}/(1-\varrho)$. The Lipschitz statement follows from $\omega_K(h;T)\leq L_K(T)h$.

Under the additional assumptions in the final part of the theorem, Corollary~\ref{cor:stieltjes-second-order}, applied with $A=Q$ and $B=K$, gives a constant $C_{Q,K,T}$ independent of $h$ such that

\begin{equation*}
\max_{0\leq m\leq N_h}\|\tau_h(m)\|_F
\leq
C_{Q,K,T}h^2.
\end{equation*}

The same resolvent estimate then gives

\begin{equation*}
\max_{0\leq m\leq N_h}\|e_h(m)\|_F
\leq
\frac{s e^{\lambda T}}{1-\varrho}C_{Q,K,T}h^2,
\end{equation*}

which proves \eqref{eq:mre-order-two}.
\end{proof}

\section{Algorithmic details}\label{app:algorithms}

\paragraph{FFT convolution.}
Let $A$ and $B$ have lengths $N_A$ and $N_B$. Choose the smallest power of two $L\geq N_A+N_B-1$ and append zero coefficients up to length $L$. Compute the $s^2$ scalar transforms of each sequence, form

\begin{equation*}
\widehat C(m)=\widehat A(m)\widehat B(m),
\qquad m\in I_L,
\end{equation*}

and apply the inverse transforms entrywise. The first $N_A+N_B-1$ coefficients give the linear convolution. Rectangular coefficient matrices are treated in the same way, provided their dimensions are compatible.

\paragraph{Newton inversion.}
Start from the sequence $B_0$ concentrated at zero with $B_0(0)=A(0)^{-1}$. If $B_r$ is correct modulo $x^m$, pad both $A$ and $B_r$ to length $2m$ and compute

\begin{equation*}
B_{r+1}=B_r*(2\ez-A*B_r)\pmod{x^{2m}}.
\end{equation*}

The residual satisfies

\begin{equation*}
\ez-A*B_{r+1}=(\ez-A*B_r)^{(2)}\pmod{x^{2m}},
\end{equation*}

so the number of correct coefficients doubles. The last step is truncated if $N$ is not a power of two.

\paragraph{Gauss--Jordan inversion.}
Form the augmented matrix

\begin{equation*}
\left[P_A(x)\mid I_s\right]
\end{equation*}

over $\C[[x]]/\langle x^N\rangle$. At column $i$, a row permutation selects a pivot whose constant coefficient is nonzero. Such a pivot exists because the constant matrix $A(0)$ is nonsingular. The pivot series is inverted by scalar Newton iteration, the pivot row is normalized, and its remaining column entries are eliminated by truncated scalar-series convolutions. After $s$ steps, the right block is $P_A(x)^{-1}$ modulo $x^N$.

We finally prove Proposition~\ref{prop:inverse-complexity}.

\begin{proof}
The coefficient recursion computes, at order $k$, a sum of $k$ matrix products. Summation over $k=1,\ldots,N-1$ gives

\begin{equation*}
\sum_{k=1}^{N-1}k=\frac{N(N-1)}{2}
\end{equation*}

products of $s\times s$ matrices and hence $\mathcal O(s^3N^2)$ arithmetic operations.

For Newton inversion, a product of two length-$m$ matrix series uses $2s^2$ scalar FFTs, $s^2$ inverse FFTs, and $m$ products of $s\times s$ matrices in the frequency domain. Its cost is therefore

\begin{equation*}
\mathcal O(s^2m\log m+s^3m).
\end{equation*}

A Newton step uses a fixed number of such products. Since the truncation length doubles, the sum over $m=1,2,4,\ldots,N$ is bounded by a constant multiple of the final-step cost and gives

\begin{equation*}
\mathcal O(s^2N\log N+s^3N).
\end{equation*}

At each of the $s$ Gauss--Jordan pivots, normalization of the pivot row uses $2s$ scalar series products and elimination from the other rows uses at most $2s(s-1)$ scalar series products. Thus the elimination uses $\mathcal O(s^3)$ scalar series products in total, together with $s$ scalar series inversions. FFT multiplication and Newton inversion of a scalar series both have cost $\mathcal O(N\log N)$, giving $\mathcal O(s^3N\log N)$. In every method, the coefficient arrays require $\mathcal O(s^2N)$ storage.
\end{proof}

\section{Distributions and parameters}\label{app:parameters}

Two discrete sojourn laws are used. If $X$ has a Gamma distribution with shape $\alpha$ and scale $\theta$, the shifted discrete Gamma mass is

\begin{equation*}
f_{\alpha,\theta}(k)
=
\mathbb P(k-1\leq X<k),
\qquad k\geq1.
\end{equation*}

The shifted Poisson mass is

\begin{equation*}
f_{\lambda}(k)
=
\exp(-\lambda)\frac{\lambda^{k-1}}{(k-1)!},
\qquad k\geq1.
\end{equation*}

Both distributions assign zero probability to time zero. The parameters are given in Table~\ref{tab:discrete-parameters}.

\begin{table}[H]
\centering
\small
\begin{tabular}{lcccc}
\toprule
Experiment & $F_{12}$ & $F_{21}$ & $F_{23}$ & $F_{31}$\\
\midrule
Horizon scaling: $(\alpha,\theta)$ & $(1.8,4)$ & $(1.6,5)$ & $(2.2,4)$ & $(1.9,3)$\\
First entrance: $\lambda$ & $8$ & $5$ & $10$ & $7$\\
\bottomrule
\end{tabular}
\caption{Sojourn-time parameters for the admissible transitions of the three-state model.}
\label{tab:discrete-parameters}
\end{table}

For the Lognormal continuous-time model, the parameters are given in Table~\ref{tab:lognormal-parameters}. Its CDF is $\Phi((\log t-\mu)/\sigma)$ for $t>0$, where $\Phi$ denotes the standard normal distribution function.

\begin{table}[H]
\centering
\small
\begin{tabular}{lccccccc}
\toprule
 & $F_{01}$ & $F_{10}$ & $F_{12}$ & $F_{20}$ & $F_{21}$ & $F_{23}$ & $F_{32}$\\
\midrule
$\mu$ & 3.94 & 3.46 & 2.12 & 1.92 & 1.61 & 2.22 & 0.94\\
$\sigma$ & 1.27 & 1.13 & 0.51 & 0.83 & 0.62 & 0.92 & 0.83\\
\bottomrule
\end{tabular}
\caption{Lognormal sojourn-time parameters for the continuous-time model.}
\label{tab:lognormal-parameters}
\end{table}
\bibliographystyle{abbrvnat}
\bibliography{references}
\end{document}